\documentclass[12pt]{article}
\usepackage{amsmath,amsfonts,amsthm}
\usepackage{mathrsfs}
\usepackage[dvips]{graphicx}
\newcommand{\N}{\mathbb{N}}
\newcommand{\Z}{\mathbb{Z}}
\newcommand{\R}{\mathbb{R}}
\newcommand{\C}{\mathbb{C}}
\newcommand{\StirCycle}[2]{ \genfrac{[}{]}{0pt}{}{#1}{#2} }
\newcommand{\StirSubSet}[2]{ \genfrac{\{}{\}}{0pt}{}{#1}{#2}_{\negthickspace\geq2} }
\newcommand{\EulerianTwo}[2]{ \left\langle \!\!\! \genfrac{ \langle}{\rangle}{0pt}{}{#1}{#2}\!\!\! \right\rangle }
\newtheorem{Th}{Theorem}
\newtheorem{Lem}{Lemma}
\newtheorem{corollary}[Th]{Corollary}
\newtheorem{remark}[Th]{Remark}
\numberwithin{equation}{section}
\begin{document}
%\tableofcontents
%\newpage
\title{Convergence in $\C$ of series \\
       for the Lambert $W$ Function}

\author{G.~A. Kalugin and D.~J. Jeffrey\\
Department of Applied Mathematics \\
The University of Western Ontario, London, Canada}
\maketitle
\begin{abstract}
We study some series expansions for the Lambert $W$ function. We show that known asymptotic series converge in both real and complex domains.
We establish the precise domains of convergence and other properties of the series, including asymptotic expressions for the expansion coefficients. We introduce an new invariant transformation of the series.
The transformation contains a parameter whose effect on the domain and rate of convergence is studied theoretically and numerically.
We also give alternate representations of the expansion coefficients, which imply a number of combinatorial identities.
\end{abstract}

%------------------------------------------------------------------------------------------------------------
\section{Introduction}
%------------------------------------------------------------------------------------------------------------
The equation $y^\alpha e^y = x$ was solved by de Bruijn and by Comtet \cite{Bruijn, Comtet}, as an asymptotic expansion:
\begin{equation}\label{eq:opening}
y = \Phi_\alpha(x)=\ln x-\alpha\ln\ln x+ \alpha u = \alpha\left(\frac{1-\tau}{\sigma}+u\right)
\ ,  \end{equation}
where
\begin{equation}\label{eq:st}
\sigma=\frac{\alpha}{\ln x}\ ,\quad \tau=\alpha\,\frac{\ln\ln x}{\ln x}\ ,
\end{equation}
and
\begin{equation} \label{eq:Comtet}
u(\sigma,\tau) =\sum_{n=1}^\infty \sum_{m=1}^n (-1)^{n-m}
\StirCycle{n}{n-m+1} \frac{\sigma^{n-m} \tau^{m}}{m!}\ ,
\end{equation}
and where $\StirCycle{n}{n-m+1}$ denotes a Stirling cycle
number~\cite{{Graham et. al},{Corless1997seqseries}}.
The function $u$ obeys the fundamental relation
\begin{equation}\label{eq:basic}
1-e^{-u}+\sigma u-\tau=0\ .
\end{equation}
By rewriting \eqref{eq:basic} in terms of $\zeta=1/(1+\sigma)$,
a second series for $u$ was obtained by~\cite{CRpaper} in terms
of 2-associated Stirling partition numbers~\cite{{Graham et. al}},
also called Stirling subset numbers or Stirling numbers of the 2nd kind.
\begin{equation} \label{eq:improvedStir}
u=\sum_{m=1}^\infty  \sum_{p=0}^{m-1}
\StirSubSet{p+m-1}{p} {(-1)^{p+m-1}} \frac{\zeta^{p+m}\tau^m}{m!}\ .
\end{equation}

Since $y^\alpha e^y=x$ can be solved in terms of the Lambert $W$ function, as
$y=\alpha W(x^{1/\alpha}/\alpha)$, the series solutions above are also series for 
the principal branch of $W$.
We recall that branches are denoted $W_k$, and of these only $W_0$ and $W_{-1}$ take real values, obeying the inequalities~\cite{Corless1996Lam}
\begin{align*}
&-1\le W_0(x) < \infty \quad \mbox{for} \quad -1/e \le x <\infty\ , \\
&-\infty < W_{-1}(x) <-1 \quad \mbox{for} \quad -1/e < x < 0\ .
\end{align*}
Here, we are concerned only with the principal branch $k=0$,
although some discussion of series expansions for branch $k=-1$ is given in section \ref{sec:minus1}.

%The fundamental relation \eqref{eq:basic} possesses a remarkable property:
%it can itself be solved in terms of the Lambert $W$ function \cite{Corless1997seqseries}
%\begin{equation} \label{eq:u(sigma,tau)}
%u=W(e^s)-\frac{1-\tau}{\sigma} \ ,
%\end{equation}
%where
%\begin{equation} \label{eq:s(sigma,tau)}
%s=s(\sigma,\tau)=\frac{1-\tau}{\sigma}-\ln \sigma \ .
%\end{equation}
%This gives a useful representation of the Lambert $W$ function \cite{Corless1997seqseries}
%\begin{equation} \label{eq:W(exp(s))}
%W(z)=W(e^s) \
%\end{equation}
%and allows to get properties of $u$ from those of $W(z)$ and vice versa.
%For example, it follows from \eqref{eq:u(sigma,tau)} that in the real case
%\begin{equation} \label{eq:Im u}
%-\pi<\Im u<\pi
%\end{equation}
%because the same is true for $W$ \cite{Corless1996Lam}.

The convergence properties of the asymptotic series \eqref{eq:Comtet} and \eqref{eq:improvedStir} were first studied in the real case in \cite{CRpaper}, and bounds
on the domains of convergence obtained.
Here we establish the precise domains of convergence of the series both on the real line and
in the complex plane. We analyze the differences in the properties of the series and find asymptotic expressions for the expansion coefficients in \eqref{eq:improvedStir}.
We also consider invariant transformations of the above series.
The transformations contain a parameter $p$ (related to $\alpha$ above) and retain the basic series structure while $p$ varies.
The parameter changes the convergence domains of the series as well as
the rates of convergence.

A third series is studied, derived from one for the Wright $\omega$
function~\cite{Wright function}.
\begin{equation} \label{eq:wEuler}
W(x)=\omega_0 + \sum_{m=1}^\infty \frac{1}{m! \sigma^m (1+\omega_0)^{2m-1}}
\sum_{k=0}^{m-1} \EulerianTwo{m-1}{k} (-1)^k \omega_0^{k+1}\ ,
\end{equation}
where $\sigma=1/\ln x$  and $\omega_0$ denotes the Omega constant $W(1)=0.56714329..$.
The notation and definition of the second-order Eulerian numbers
may be found in \cite{{Graham et. al},{Corless1997seqseries}}.
We give three new representations of the expansion coefficients of
this series as well as their asymptotic estimates.
It is also shown that the series \eqref{eq:improvedStir} can also be expressed using the second-order Eulerian numbers. Some combinatorial identities follow from the different forms
for coefficients, including the Carlitz-Riordan identities.

%----------------------------------------------------------------------------------------------------------------------
\section{First series expansion \eqref{eq:Comtet}}
%-------------------------------
It is convenient to consider separately the case in which $x>1$, making $\sigma$, $\tau$ and $u$ all real. 
This was the case studied in~\cite{CRpaper}, where bounds on the convergence domain 
were given, and is the subject of the next subsection.
\subsection{Convergence for real values} \label{sec:ComtetReal}
%----------------------------------------------------------------------------------
We begin by stating a convergence criterion in terms of the variables $\sigma$ and $\tau$.

\begin{Th} \label{Th:ConvComtetSigmaTau}
The domain of convergence of the series \eqref{eq:Comtet} is defined by the inequality
\begin{equation} \label{eq:ConvComtetSigmaTau}
\ln\sigma<1-\frac{\tau}{\sigma}+\Re W_{-1}\left(-e^{\frac{\tau}{\sigma}-1}\right)\ .
\end{equation}
\end{Th}

\begin{proof}
We rewrite the fundamental relation (\ref{eq:basic}) by
introducing $\lambda=\tau/\sigma$ to play the role of a parameter.
\begin{equation} \label{eq:G(u)}
G_\lambda(\sigma,u) = 1 - e^{-u} + \sigma u - \sigma\lambda=0\ .
\end{equation}
By the implicit function theorem \cite{Markushevich}, for fixed $\lambda\in\R$, equation \eqref{eq:G(u)} determines a function
\begin{equation} \label{series:sigma}
u_\lambda(\sigma)=\sum_{m}c_m(\lambda)\sigma^m
\end{equation}
with initial condition $u_\lambda(0)=0$ in a domain where $\partial G_\lambda(\sigma,u)/\partial u = e^{-u} + \sigma \neq 0$.
The radius of convergence
of this power series equals the distance from the origin in the complex $\sigma$-plane to the closest singular point~\cite{Titchmarsh},\cite[p.175, Theorem 4.3.2]{Antimirov et. al}.
The critical points in the complex $\sigma$-plane, 
where $\partial G_\lambda(\sigma,u)/\partial u = 0$, satisfy the relations, after using 
\eqref{eq:G(u)},
\begin{eqnarray} \label{u:crit}
e^{-u} &= & -\sigma\ ,\\
\lambda -u-1& = & 1/\sigma\ ,\label{eq:ulambdasigma}\\
-e^{\lambda-1} \sigma &= & e^{1/\sigma} \label{eq:lambdasigma} \ .
\end{eqnarray}
From these, $\sigma$ can be written in terms of the Lambert $W$ function
\begin{equation} \label{def1:sigma}
\sigma_m=1/W_m(-e^{\lambda-1}) \ ,
\end{equation}
where the $m$th root is defined by the $m$th branch of $W$.
Further, by \eqref{eq:ulambdasigma}, $\Im(1/\sigma)=\Im(-u)$, and for the principal
branch $\Im(u)\in (-\pi,\pi)$.
Thus we conclude that there are only two acceptable values for $m$, i.e. $m=-1,0$.
We now write \eqref{eq:lambdasigma} as
\begin{equation} \label{def2:sigma}
\sigma_m=-\exp\left\{1-\lambda+W_m(-e^{\lambda-1})\right\} \quad (m=-1,0) \ ,
\end{equation}
which defines the domain of convergence by the inequality
\begin{equation}
\left|\sigma\right|<\min_{m\in\left\{-1,0\right\}}\left|-\exp\left\{1-\lambda+W_m(-e^{\lambda-1})\right\} \right|\ . \notag
\end{equation}
or
\begin{equation} \label{ineq:trans}
\ln\left|\sigma\right|<1-\Re\lambda+\min_{m\in\left\{-1,0\right\}}\Re W_m(-e^{\lambda-1})\ .
\end{equation}
Since $\Re W_{-1}(x)\leq\Re W_0(x)$ for all $x\in\R$ \cite{Corless1996Lam},
after substituting $\lambda=\tau/\sigma$ the condition \eqref{eq:ConvComtetSigmaTau} follows.
\end{proof}

To express the condition \eqref{eq:ConvComtetSigmaTau} of convergence of the series \eqref{eq:Comtet}
in terms of independent variable $x$ in \eqref{eq:st} it is convenient to prove the following lemma.

\begin{Lem} \label{lemma}
Solution of inequality $\Re W_{-1}(x)>a$ for $x<0$, where $a$ is constant, is given by
\begin{equation} \label{eq:lemma}
x<b =
\begin{cases}
          ae^a,       & a\leq-1\\
-e^a\eta_0\csc\eta_0, & a>-1
\end{cases}
\end{equation}
where $\eta_0\in\left(0,\pi\right)$ is the root of equation $\eta_0\cot\eta_0=-a$.
\end{Lem}

\begin{proof}
We set $W_{-1}(x)=\xi+i\eta$ for real negative $x$ where $\xi\leq -1, \eta=0$ for $-1/e\leq x<0$
and $\xi> -1, -\pi<\eta<0$ for $x<-1/e$.
Then $\xi,\eta$ obey~\cite{Corless1996Lam}
\begin{equation} \label{sys1:x}
x=e^\xi(\xi\cos\eta-\eta\sin\eta), \quad 0=e^\xi(\eta\cos\eta+\xi\sin\eta)  \ .  \notag
\end{equation}
From these equations, one can find the dependence of $\xi$ on $x$ explicitly for $-1/e\leq x<0$
\begin{equation} \label{eq:eta=0}
\xi=W_{-1}(x)
\end{equation}
and parametrically for $x<-1/e$
\begin{equation} \label{eta<0:x}
x=-\eta\csc(\eta)e^{-\eta\cot\eta} \ ,
\end{equation}
\begin{equation} \label{eta<0:xi}
\xi=-\eta\cot\eta \ ,
\end{equation}
where $-\pi<\eta<0$.

Now we consider inequality $\xi>a$ in two cases comparing $a$ with value $-1$.
\\When $a\leq-1$ the inequality $\xi>a$ holds for all $x<-1/e$
because in this case $\xi > -1$ by \eqref{eta<0:xi}. For $-1/e\leq x<0$
we solve inequality $W_{-1}(x)>a$ due to \eqref{eq:eta=0} with the result $-1/e\leq x<ae^a$.
Thus $\xi>a$ for $x<ae^a$.
When $a>-1$ the inequality $\xi>a$ can have a solution only for $x<-1/e$
because $\xi\leq -1$ for the rest $x$. According to \eqref{eta<0:x} and \eqref{eta<0:xi}
the solution is given by $x<-\eta_0\csc(\eta_0)\exp(-\eta_0\cot\eta_0)$
where $\eta_0\in(-\pi,0)$ satisfies the equation $-\eta_0\cot\eta_0=a$
due to which the solution can also be written as
$x<-e^a\eta_0\csc\eta_0$ and $\eta_0\in(0,\pi)$.
Joining both cases, the lemma follows.
\end{proof}
Note that in the formula \eqref{eq:lemma}, when $a>-1$ but $a\neq0$, we can also write $b=ae^a/\cos\eta_0$.

\begin{Th} \label{Th:ConvComtetReal}
The series \eqref{eq:Comtet} is convergent when
\begin{equation} \label{eq:ConvComtetReal}
x>x_\alpha =
\begin{cases}
\left(e/\alpha\right)^\alpha, & 0<\alpha\leq1 \\
 e^{\alpha\eta_0\csc\eta_0}, & \alpha>1
\end{cases}
\end{equation}
where $\eta_0$ satisfies the equation $\eta_0\cot\eta_0=1-\ln\alpha \mspace{8mu} (0<\eta_0<\pi)$,
and the series is divergent otherwise.
\end{Th}

\begin{proof}
We consider the condition of convergence of the series \eqref{eq:Comtet}
established by Theorem~\ref{Th:ConvComtetSigmaTau}
in the real case, i.e. when $\alpha>0$ and $x>1$.
Substituting the expressions \eqref{eq:st} in \eqref{eq:ConvComtetSigmaTau} we obtain
\begin{equation} \label{ineq:Re}
\Re W_{-1}\left(-\frac{\ln x}{e}\right)>\ln\alpha-1 \ .
\end{equation}
Applying Lemma \ref{lemma} to the inequality \eqref{ineq:Re} we come to \eqref{eq:ConvComtetReal},
where $x_\alpha>1$, which justifies the assumption $x>1$.
Thus the theorem is completely proved.
\end{proof}

{\it Note}. The statement of Theorem \ref{Th:ConvComtetReal} was independently reported
by A.J.E.M. Janssen and J.S.H. van Leeuwaarden \cite{JanssenLeeuwaarden}.

\begin{remark}
In the formula \eqref{eq:ConvComtetReal}, when $\alpha>1$ but $\alpha\neq e$, we can also write $x_\alpha=\left(e/\alpha\right)^{\alpha\sec\eta_0}$.
\end{remark}

\begin{remark}\label{rem:x(alpha)}
It follows from \eqref{eq:ConvComtetReal} that $x_\alpha\rightarrow1$ as $\alpha\rightarrow0$
and $x_\alpha\rightarrow\infty$ as $\alpha\rightarrow\infty$.
In addition, one can show that a function of $\alpha$ defined by \eqref{eq:ConvComtetReal} is monotone increasing.
Therefore, the larger $\alpha$, the less the domain of convergence of the series \eqref{eq:Comtet}.
\end{remark}

\begin{corollary} \label{cor:x>e}
Series \eqref{eq:Comtet} for $\alpha=1$, i.e. for $W$ function, is convergent when
\begin{equation} \label{eq:Remark2.2}
\Re\left[W_{-1}\left(-\frac{\ln x}{e}\right)\right]>-1 \ ,
\end{equation}
which is equivalent to $x>e$,
and divergent when the opposite inequalities hold.
\end{corollary}
\begin{proof}
Follows immediately from \eqref{ineq:Re} and \eqref{eq:ConvComtetReal}.
\end{proof}
\begin{remark}
In terms of variable $\sigma=1/\ln x$ series \eqref{eq:Comtet} is convergent for $0<\sigma<1$.
\end{remark}

We now prove a statement, relating to divergence of the series \eqref{eq:Comtet},
which was found by us earlier than the conditions \eqref{eq:ConvComtetReal} but
unlike Theorem \ref{Th:ConvComtetReal} concerning positive $\alpha$
it deals with any $\alpha\neq0$. In addition, the statement demonstrates
an interesting application of the ratio test to the series \eqref{eq:Comtet}.

\begin{Th} \label{Th:DivComtet}
The series \eqref{eq:Comtet} is divergent at least for
\begin{equation}\label{eq:thdiv}
 e^{-\left|\alpha \right|}<x<e^{b\left|\alpha \right|}
\end{equation}
{where}
\medskip

$$
b=
  \left\{
  \begin{array}{cc}
   W \left(\displaystyle 1/\left|\alpha \right| \right) & \mbox{when $\left|\alpha \right|< 1/e $ },\\[10pt]
   1                                              & \mbox{when $\left|\alpha \right|\geq 1/e $ }.
  \end{array}
  \right.
$$
\end{Th}
\medskip

\begin{proof} Changing indices for summing the expansion \eqref{eq:Comtet} can be written through a double series \cite{CRpaper}
%\textit{Proof}. Changing indices for summing the expansion \eqref{eq:Comtet} can be written through a double series \cite{CRpaper}
\begin{equation}\label{doublesum}
u=\sum_{m=1}^\infty \sum_{l=0}^\infty c_{m,l},
\end{equation}
where
$$
c_{m,l}=c_{m,l}(\sigma,\tau)=\frac{(-1)^l}{m!}
\StirCycle{l+m}{l+1}
\sigma ^l \tau^{m}
$$
For the column-series $\sum_m c_{m,l}$ the ratio test gives
$$
\lim\limits_{m \to\infty}\left| \frac{c_{m+1,l}}{c_{m,l}} \right|
=\left| \tau \right| \lim\limits_{m \to\infty} \frac
{ \displaystyle\StirCycle{l+m+1}{l+1} }
{ \displaystyle (m+1)\StirCycle{l+m}{l+1} }
= \left| \tau \right|
$$
as according to \cite{Abramowitz}
$$
\lim\limits_{p \to\infty}
\frac
{ \displaystyle\StirCycle{p+1}{l+1} }
{ \displaystyle p \StirCycle{p}{l+1} }
= 1 \qquad \mbox{for fixed } l
$$
in our notations.
\\For the row-series $\sum_l c_{m,l}$ we have
$$
\lim\limits_{l \to\infty}\left| \frac{c_{m,l+1}}{c_{m,l}} \right|
=\left| \sigma \right| \lim\limits_{l \to\infty}\frac
{ \displaystyle\StirCycle{l+m+1}{l+2} }
{ \displaystyle\StirCycle{l+m}{l+1} }
= \left| \sigma \right|
$$
because by \cite{Abramowitz}
$$
\lim\limits_{l \to\infty} \frac
{ \displaystyle\StirCycle{l+m}{l+1} }
{ (l+1)^{2m-2} }
= \frac{1}{2^{m-1}(m-1)!} \qquad \mbox{for fixed } m.
$$
According to \cite[Theorem 2.7]{LimayeZeltser} the series \eqref{doublesum} (and therefore \eqref{eq:Comtet}) is divergent when $\left| \tau \right|>1$ or $\left| \sigma \right|>1$. Expressing these inequalities in terms of $x$ by (\ref{eq:st}) and uniting the obtained sets we come to the stated inequality \eqref{eq:thdiv}.
\end{proof}

By Theorem \ref{Th:DivComtet} for $\alpha=1$ the series \eqref{eq:Comtet} is divergent at least for $e^{-1}<x<e$, which is consistent with Corollary \ref{cor:x>e}.

For comparison, the curves described by equations \eqref{eq:ConvComtetReal} and \eqref{eq:thdiv} are shown in Figure~\ref{Fig:ComtetReal} by solid and dash-dotted lines respectively, together with the bound given in~\cite{CRpaper}.
\begin{figure}[ht]
  \noindent\centering
  {
  \includegraphics[width=80mm]{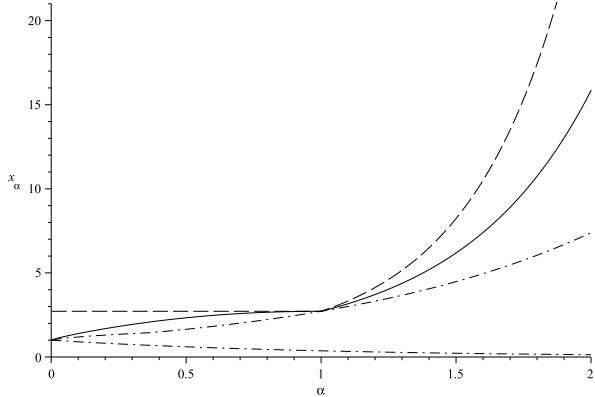}
  }
  \caption{Boundary of domain of convergence of series \eqref{eq:Comtet} given by \eqref{eq:ConvComtetReal}
           is depicted by solid line; dashed and dash-dotted lines are described by equations given in reference~\cite{CRpaper}
           and \eqref{eq:thdiv} respectively.}
  \label{Fig:ComtetReal}
\end{figure}

%------------------------------------------------------------------------------------------------------------------------
\subsection{Convergence in complex case} \label{sec:ComtetComplex}
%------------------------------------------------------------------------------------------------------------------------
In this subsection we consider equations \eqref{eq:opening}, \eqref{eq:st} only for $\alpha=1$ (under the same relation \eqref{eq:basic}) and derive convergence conditions for the series \eqref{eq:Comtet} in the complex case using the results obtained in Section \ref{sec:ComtetReal} in the real case and providing a continuous continuation of the latter. We set
\begin{equation} \label{st:complex}
\sigma=1/{\ln z}, \tau={\ln\ln z}/{\ln z} \ ,
\end{equation}
where $z=x+iy$ is a complex variable and $\ln z$ denotes the principal branch of the natural logarithm. Then the right-hand side of the series \eqref{eq:Comtet} represents a function of the complex variable $z$ and the following theorem holds.

\begin{Th} \label{Th:ConvComtetComplex}
The domain of convergence of the series \eqref{eq:Comtet} in the complex $z$-plane is \mbox{defined by}
\begin{equation} \label{eq:ConvComtetComplex}
\Re W_m\left(-\frac{\ln z}{e}\right)>-1,
\end{equation}
where the branch $W_m$ is chosen as follows
$$
m=
\begin{cases}
-1, & -\pi < \arg z \leq 0 \\
1, & 0 < \arg z \leq \pi
\end{cases}
$$
\end{Th}

\begin{proof}
Repeating the proof of Theorem \ref{Th:ConvComtetSigmaTau} under assumption $\lambda\in\C$
we come to an equation which is different from \eqref{ineq:trans} only by that $m\in\Z$.
Then substituting \eqref{st:complex} in there we obtain (cf. \eqref{eq:Remark2.2})
\begin{equation} \label{ineq:Z}
\min_{m\in\Z}\Re W_m\left(-\frac{\ln z}{e}\right)>-1 \ .
\end{equation}
Now we cut the complex $z$-plane along the negative real axis and set $\arg z\in(-\pi,\pi]$.
We consider inequality \eqref{ineq:Z} in domain $D=\left\{z\in\C|\,-\pi<\arg z\leq0\right\}$
and assume that there exists some value $m=q$ such that
the domain of convergence in $D$ is defined by equation
\begin{equation} \label{eq:D}
\Re W_q\left(-\ln z/e\right)>-1
\end{equation}
and its continuous boundary $\mathscr{L}$ is given by
\begin{equation} \label{eq:Z}
\Re W_q\left(-\ln z/e\right)=-1 \ .
\end{equation}
The domain of convergence found in real case is defined in a similar way.
Specifically, in domain $\left\{z\in\R|\,z>0\right\}$ we have
$q=-1$ and the boundary $z=e$ (see Corollary \ref{cor:x>e}).
We require that in the limiting case $\arg z\rightarrow0-$
equation \eqref{eq:D} become equation \eqref{eq:Remark2.2} and
show that there is an unique value $q=-1$ satisfying this requirement.
(If there were several such values of $q$, it would mean that
the boundary $\mathscr{L}$ is composed of several pieces of different curves,
and to identify them one should reduce domain $D$, i.e. consider its subdomains.)
Separating the real and imaginary parts of $W_q\left(-\ln z/e\right)=-1+i\eta$ in the usual way, we find
\begin{subequations}
\begin{equation} \label{eq:argZ}
\sin\eta-\eta\cos\eta=\arg z
\end{equation}
\begin{equation} \label{eq:lnZ}
\cos\eta+\eta\sin\eta=\ln\left|z\right|
\end{equation}
\end{subequations}
These equations describe a set of the boundary points which can be found in the following way.
Given a value for $\arg z$ one can find $\eta$ from \eqref{eq:argZ}
which being substituted in \eqref{eq:lnZ} yields the corresponding value of $\ln\left|z\right|$.
However, for fixed $\arg z\in(-\pi,0]$ the equation \eqref{eq:argZ} has
an infinite number of solutions.
We select a solution to provide a continuous transition to the real case when $\arg z\rightarrow0-$
and when the boundary of the domain of convergence is defined by $\Re W_{-1}(-\ln z/e)=-1$ (cf. \eqref{eq:Remark2.2}).
An elementary analysis of the equation \eqref{eq:argZ} shows that to meet these requirements
one needs to choose a solution  of this equation from the interval $\eta\in(-\pi,0]$
and set $q=-1$ in \eqref{eq:Z}.
Since by \eqref{eq:argZ} such solution exists if and only if $z\in D$, the above assumption is approved
and the domain of convergence in $D$ is described by \eqref{eq:D} with $q=-1$, i.e.
\begin{equation}
\Re W_{-1}(-\ln z/e)>-1 \ . \notag
\end{equation}
Due to the near conjugate symmetry property of $W$ function \cite{Corless1996Lam}, i.e. $W_k(z)=\overline{W_{-k}(\bar{z})}$
when $z$ is not on the branch cut, we obtain the convergence condition
$\Re W_1(-\ln z/e)>-1$ in the domain $\left\{z\in\C|\,0<\arg z\leq\pi\right\}$.
Thus the theorem is completely proved.
\end{proof}

\begin{remark}\label{Remark4.1}
The 'branch splitting' in the proved formula \eqref{eq:ConvComtetComplex} is due to
the branch choices for the Lambert $W$ function and similar to the effect that occurs in the series
for $W$ about the branch point \cite[Sec.\,3]{Corless1997seqseries}.
\end{remark}

\begin{remark}\label{Remark4.2}
The inequality opposite to \eqref{eq:ConvComtetComplex} defines the domain
where the series \eqref{eq:Comtet} is divergent. This domain is finite (it encloses the origin $z=0$)
and contains a subdomain defined by inequality $\left|\sigma\right|>1$.
Therefore, unlike the real case (see Corollary \ref{cor:x>e}) in the complex case
the condition $\left|\sigma\right|<1$ is only necessary but not sufficient
for convergence of the series \eqref{eq:Comtet}.
\end{remark}

%-----------------------------------------------------------------------------------------------------------------------------
\section{Series \eqref{eq:improvedStir}}
%-----------------------------------------------------------------------------------------------------------------------------
\subsection{Convergence in real case} \label{RealSigma}
%-----------------------------------------------------------------------------------------------------------------------------
We regard the expansion \eqref{eq:improvedStir} as a power series around $\tau = 0$ where variable $\sigma$ plays a role of a parameter.

\begin{Th} \label{Th:ConvReal}
For $\alpha > 0$ and $\sigma > 0$, the radius of convergence of the series \eqref{eq:improvedStir} is exactly
\begin{equation} \label{ThImprovedReal1}
\tau_\ast(\sigma)=\left|1+\sigma-\sigma\ln \sigma\pm i\pi\sigma\right|.
\end{equation}
which is equiavalent to the condition of convergence of the series \eqref{eq:improvedStir} as
\begin{equation} \label{ThImprovedReal2}
\left|\sigma(\ln\alpha-\ln\sigma)\right|<\sqrt{(1+\sigma-\sigma\ln\sigma)^2 + \pi^2\sigma^2}.
\end{equation}
\end{Th}

\begin{proof}
We rewrite the fundamental relation (\ref{eq:basic}) in the form of equation
\begin{equation} \label{ThImprovedReal3}
F_\sigma(\tau,u)=0,
\end{equation}
where
\begin{equation}\label{ThImprovedReal4}
F_\sigma(\tau,u) = 1 - e^{-u} + \sigma u - \tau \ ,
\end{equation}
and analyse this equation similarly to that in the proof of Theorem \ref{Th:ConvComtetSigmaTau}.
By Implicit Function Theorem \cite{Markushevich} the equation \eqref{ThImprovedReal3} determines
a function $u_\sigma(\tau)=\sum_{m}c_m(\sigma)\tau^m$ with initial condition
$u_\sigma(0)=0$ in a domain where $\partial F_\sigma(\tau,u)/\partial u = e^{-u} + \sigma \neq 0$.
The initial condition is justified by $\partial F_\sigma(0,0)/\partial u = 1+\sigma \neq0$.
Since the critical points are defined by the same equation as in Theorem \ref{Th:ConvComtetSigmaTau},
they are given by \eqref{u:crit}
and the corresponding values of $\tau$ are
\begin{equation}\label{ThImprovedReal5}
 \tau_\ast^{(k)}=1 - e^{-u_\ast^{(k)}} + \sigma u_\ast^{(k)} = 1 + \sigma -\sigma \ln \sigma + i\pi\sigma(2k-1), k\in\Z  %\eqno (12)
\end{equation}

The radius of convergence is equal to the distance from the origin in the complex $\tau$-plane to the closest singular point
\cite[Theorem 4.3.2]{Antimirov et. al}. Among the critical points \eqref{ThImprovedReal5}
there are two the nearest to the origin equidistant points which correspond to $k=0$ and $k=1$:
\begin{subequations} \label{singtau}
\begin{equation}\label{singtau0}
 \tau_\ast^{(0)} = 1 + \sigma -\sigma \ln \sigma - i\pi\sigma,
\end{equation}
\begin{equation}\label{singtau1}
 \tau_\ast^{(1)} = 1 + \sigma -\sigma \ln \sigma + i\pi\sigma.
\end{equation}
\end{subequations}
The corresponding values of $u_\ast^{(k)}$ are
\begin{subequations} \label{singu}
\begin{equation}\label{singu0}
 u_\ast^{(0)} = -\ln \sigma - i\pi,
\end{equation}
\begin{equation}\label{singu1}
 u_\ast^{(1)} = -\ln \sigma + i\pi.
\end{equation}
\end{subequations}

Since the expansion coefficients of the series \eqref{eq:improvedStir} are real, the closest singularities can appear as a conjugate pair only \cite{Hunter&Guerrieri}. Based on the Weierstrass's preparation theorem \cite{{Markushevich}, {Adachi}} we will show that the points \eqref{singtau}  are singular, each corresponding to a square-root branch point of function $u=u_\sigma(\tau)$ in the complex $\tau$-plane. We will also find a behavior of function $u=u_\sigma(\tau)$ near the points \eqref{singtau} used then for a study of an asymptotic behaviour of the expansion coefficients of the series \eqref{eq:improvedStir}.

Let us consider, for example, point $\tau=\tau_\ast^{(0)}$. Expanding the left-hand side of equation \eqref{ThImprovedReal3} into a Taylor series near the point $S\left( \tau_\ast^{(0)},u_\ast^{(0)} \right)$ we obtain
$$
F_\sigma(S)+\frac{\partial F_\sigma(S)}{\partial \tau} \left( \tau-\tau_\ast^{(0)} \right) + \frac{\partial F_\sigma(S)}{\partial u} \left( u-u_\ast^{(0)} \right) + \frac{\partial^2 F_\sigma(S)}{\partial {\tau}^2}\frac{ \left( \tau-\tau_\ast^{(0)} \right) ^2 }{2}
$$
$$
 + \frac{\partial^2 F_\sigma(S)}{\partial \tau \partial u} \left( \tau-\tau_\ast^{(0)} \right) \left( u-u_\ast^{(0)} \right) + \frac{\partial^2 F_\sigma(S)}{\partial u^2}\frac{ \left( u-u_\ast^{(0)} \right) ^2 }{2} + \dots = 0,
$$

where dots denote the skipped terms of the higher order.
Since

$$
F_\sigma(S)=0, \frac{\partial F_\sigma(S)}{\partial u}=0, \frac{\partial^2 F_\sigma(S)}{\partial u^2}=-\exp\left(-u_\ast^{(0)}\right), \mbox{ and   } \frac{\partial F_\sigma}{\partial \tau}\equiv-1
$$\\
the last equation becomes
$$
-\left( \tau-\tau_\ast^{(0)} \right)-\exp\left(-u_\ast^{(0)}\right)\frac{ \left( u-u_\ast^{(0)} \right) ^2 }{2} + \dots = 0.
$$
\\
Thus, in accordance with the Weierstrass's preparation theorem \cite[p.111]{Markushevich}, equation \eqref{ThImprovedReal3} is locally equivalent to the equation
$$
\left( \tau_\ast^{(0)}-\tau \right) \sim \exp\left(-u_\ast^{(0)}\right)\frac{ \left( u-u_\ast^{(0)} \right) ^2 }{2}.
$$
It follows that at $\tau=\tau_\ast^{(0)}$ function $u=u_\sigma(\tau)$ has a singularity corresponding to a square-root branch point as near this point
$$
 u \sim u_\ast^{(0)} \pm \exp\left(\frac{u_\ast^{(0)}}{2}\right) \sqrt{2 \left( \tau_\ast^{(0)}-\tau \right)}
$$
or substituting \eqref{singu0}
\begin{equation}\label{u+-}
u \sim -\ln\sigma-i\pi\pm i\sqrt{\frac{2\tau_\ast^{(0)}}{\sigma}} \left( 1-\frac{\tau}{\tau_\ast^{(0)}} \right)^\frac{1}{2}.
\end{equation}
It is not difficult to show that if we consider the values of the function \eqref{u+-}
in the interior of the circle of radius \eqref{ThImprovedReal1} remaining in the vicinity
of $\tau=\tau_\ast^{(0)}$ then the function \eqref{u+-} taken with the plus sign only
satisfies the condition $-\pi<\Im u<0$,
which corresponds to $\Im u_\ast^{(0)}=-\pi<0$ at point $\tau=\tau_\ast^{(0)}$ itself by \eqref{singu0}.
Moreover, since in the mentioned vicinity $-\pi<\Im\tau/\sigma<0$,
we have $-\pi<\Im W<\pi$,
which corresponds to the principal branch of $W$ function \cite{Corless1996Lam}.
Thus we come to conclusion that the function $u=u_\sigma(\tau)$
behaves near the singularity \eqref{singtau0} like
\begin{subequations}\label{behsing}
\begin{equation}
u \sim -\ln\sigma-i\pi + i\sqrt{\frac{2\tau_\ast^{(0)}}{\sigma}} {\left(1-\frac{\tau}{\tau_\ast^{(0)}}\right)}^\frac{1}{2} \mbox{ as }    \tau\rightarrow\tau_\ast^{(0)}.
\end{equation}
One can show in a similar way that near the singularity \eqref{singtau1} the function $u=u_\sigma(\tau)$ behaves like
\begin{equation}
u \sim -\ln\sigma+i\pi - i\sqrt{\frac{2\tau_\ast^{(1)}}{\sigma}} {\left(1-\frac{\tau}{\tau_\ast^{(1)}}\right)}^\frac{1}{2} \mbox{ as }    \tau\rightarrow\tau_\ast^{(1)}.
\end{equation}
\end{subequations}
Thus the points \eqref{singtau} are singular and we immediately obtain expression \eqref{ThImprovedReal1}; the inequality \eqref{ThImprovedReal2} follows from \eqref{ThImprovedReal1} as $\tau=-\sigma(\ln\sigma-\ln\alpha)$ due to (\ref{eq:st}).
The theorem is completely proved.
\end{proof}

\begin{remark} \label{Remark5.1}
From the values \eqref{singtau} and \eqref{singu}, we find $W(x)=-1$ for both $k=0$ and $k=1$. Although it is well-known that this value of the Lambert $W$ function corresponds to its branch point and asymptotics \eqref{behsing} can be obtained immediately from the results in \cite{{Corless1996Lam},{Corless1997seqseries}}, we derived these asymptotic formulae to demonstrate a method based on the Weierstrass's preparation theorem.
\end{remark}

The inequality \eqref{ThImprovedReal2} can be written in the form $\left|\ln\alpha-\ln\sigma\right|<g(\sigma)$, where
\begin{equation}\label{eq:g(sigma)}
g(\sigma)=\sqrt{\pi^2+(1+1/\sigma-\ln\sigma)^2} \ ,
\end{equation}
and solved with respect to $\alpha$. Then the following theorem follows.

\begin{Th}\label{th:AlphaDomains}
Let $\sigma_c=1.059945...$ is the unique root of equation $g(\sigma)\sqrt{(1+1/\sigma)^2-1}/(1+1/\sigma)=\pi$
and $\alpha_c=\sigma_c\exp(g(\sigma_c))=41.349171...$, where function $g(\sigma)$ is defined by \eqref{eq:g(sigma)}.
Then the domain of convergence of series \eqref{eq:improvedStir} depending on $\alpha$ is defined as follows.
\begin{enumerate}
\renewcommand{\labelenumi}{ \upshape{(\roman{enumi})} }
\item For $0<\alpha<e$, the series \eqref{eq:improvedStir} is convergent when $0<\sigma<\sigma_\alpha$,
where $\sigma_\alpha$ is the only root of the equation
\begin{equation} \label{eq:SigmaAlpha}
 \ln\sigma - g(\sigma) = \ln\alpha,
\end{equation}
which is equivalent to
\begin{equation} \label{eq:alpha}
x>x_\alpha=e^{\alpha/\sigma_\alpha} \ .
\end{equation}
The series is divergent when $\sigma>\sigma_\alpha$ or $1<x<x_\alpha$.
\item For $e<\alpha<\alpha_c$, the series \eqref{eq:improvedStir} is convergent for any $\sigma>0$ or for any $x>1$.
\item For $\alpha>\alpha_c$, the series \eqref{eq:improvedStir} is convergent when $\sigma<\mu_\alpha$ or $\sigma>\nu_\alpha$,
where $\mu_\alpha$ and $\nu_\alpha$ are roots of equation $\left|\ln\sigma-\ln\alpha\right|=g(\sigma)$,
which is equivalent to $x>e^{\alpha/\mu_\alpha}$ or $x<e^{\alpha/\nu_\alpha}$. The series is divergent
when $\mu_\alpha<\sigma<\nu_\alpha$ or $e^{\alpha/\nu_\alpha}<x<e^{\alpha/\mu_\alpha}$.
\end{enumerate}
\end{Th}
\begin{proof}
We give details only for the first part (i).
In a particular case $\alpha<\sigma$, equation \eqref{ThImprovedReal1}
can be written as \eqref{eq:SigmaAlpha} with $g(\sigma)$ defined by \eqref{eq:g(sigma)}.
The left-hand side of equation \eqref{eq:SigmaAlpha}, being a monotone increasing function for positive $\sigma$,
goes to $-\infty$ and 1 when $\sigma$ tends to 0 and $\infty$ respectively.
Therefore, for $0<\alpha<e$ the equation has the unique solution.
Applying Theorem~\ref{Th:ConvReal} the theorem follows.
\end{proof}

\begin{corollary} \label{cor:5.1}
The series \eqref{eq:improvedStir} for $W$ function is convergent for $0<\sigma<\sigma_1$,
where $\sigma_1 = 224.790951...$ is the only root of the equation
\begin{equation}\label{eq:realsigma}
\left|-\sigma\ln \sigma\right|=\sqrt{(1+\sigma-\sigma\ln \sigma)^2 + \pi^2\sigma^2},
\end{equation}
which is equivalent to
\begin{equation}\label{eq:x}
x>x_1=e^{1/\sigma_1}=1.004458... \ .
\end{equation}
The series \eqref{eq:improvedStir} is divergent for $\sigma>\sigma_1$ or for $1<x<x_1$.
\end{corollary}
\begin{proof}
Follows from Theorem \ref{th:AlphaDomains}(i) for $\alpha=1$.
\end{proof}

\begin{remark} \label{rem:x>1.0044}
In terms of the variable $x$ the series \eqref{eq:improvedStir} for $W$ function is convergent
for $x>x_1>1$ rather than for $x>1$ though $x_1$ is very close to unit.
\end{remark}

\begin{remark} \label{rem:x=x1}
Elementary analysis of equation \eqref{eq:SigmaAlpha} with substituting $\sigma=\alpha/\ln x(\alpha)$ therein shows
that $x(\alpha)>1$ for $0<\alpha<e$ and has one maximum $e^{e^{-\pi}}=1.044161...$ at point $\alpha=e^{-\pi}/W(1/e)=0.155186...$.
This means the dependence $x(\alpha)$ to be a very weak and
$x_\alpha=x(\alpha)$ can be evaluated with a good precision (with the relative error less than 5\,$\%$) by a simple approximate equality $x_\alpha\approx x_1$ ($0<\alpha<e$) which becomes accurate when $\alpha=1$.
\end{remark}

\begin{remark} \label{Remark5.3}
The solution $\sigma=\sigma_1$ of the equation \eqref{eq:realsigma} is much more than unit and can be found approximately with a good precision. Specifically, taking square of the both sides of \eqref{eq:realsigma} and leaving the main terms we obtain $\sigma^2-2\sigma^2\ln\sigma + \pi^2\sigma^2 - 2\sigma\ln\sigma \approx 0$. Searching for a solution of the approximate equation in the form $\sigma=\exp(\frac{1+\pi^2}{2})(1+\delta)$, where the exponential factor is an exact solution of the approximate equation with neglected last term and a correction term $\delta$ is to be determined, we obtain an approximate value in deficit $\sigma_1\approx\exp(\frac{1+\pi^2}{2})-\frac{1+\pi^2}{2}=223.8126969...$.
Taking into consideration of the terms of higher powers in $\delta$ in a similar way, one can obtain a more accurate value.
\end{remark}

\begin{remark} \label{Remark5.2}
For fixed $\sigma>0$ and $\alpha=1$, the singular points \eqref{singtau} of function $u_\sigma(\tau)$ correspond to
those of the Wright function $\omega=\omega(s)$,
which are $s_\ast=\xi \pm i\pi$, where $\xi\leq-1$ (see subsection \ref{Wrightfunction}).
Indeed, we have $\tau_\ast=1-\sigma\ln\sigma-\sigma s_\ast$,
i.e. $\tau_\ast=1-\sigma\ln\sigma-\xi\sigma \mp i\pi\sigma$.
Since $\Re\tau_\ast$ has the minimum at $\xi=-1$, the closest singular points are defined by equations, which are exactly \eqref{singtau}.
\end{remark}

\begin{remark} \label{Remark5.4}
The convergence condition \eqref{ThImprovedReal2} has a clear geometrical interpretation in $(\sigma,\tau)$-plane. For example, for $\alpha=1$ one can show that in accordance with the inequality \eqref{ThImprovedReal2}, when $\sigma<\sigma_1$ the curve $L$ described by $\tau=-\sigma\ln\sigma$ is located inside the region $S$ bounded by curves $\tau=\pm\sqrt{(1+\sigma-\sigma\ln \sigma)^2 + \pi^2\sigma^2}$ , which expresses the condition of convergence of the series \eqref{eq:improvedStir}. However, at point $\sigma=\sigma_1$ the curve $L$ leaves the region $S$ through the lower boundary curve that can be described for large $\sigma$ by the asymptotic expression
\begin{equation}
\tau(\sigma)=-\sqrt{(1+\sigma-\sigma\ln \sigma)^2 + \pi^2\sigma^2}=-\sigma\ln \sigma + \sigma - \frac{1+\pi^2}{2} \frac{\sigma}{\ln \sigma} + 1 + O\left(\frac{1}{\ln  \sigma}\right). \notag
\end{equation}
It follows that afterwards the curve $L$ remains below the lower boundary of $S$, which corresponds to the divergence of the series \eqref{eq:improvedStir} for $\sigma>\sigma_1$.
\end{remark}

Now we consider case $\sigma<0$, which should be done carefully as by Implicit Function Theorem it should be $\partial F_\sigma(0,0)/\partial u \neq 0$ due to the initial condition $u_\sigma(0)=0$ and therefore the value $\sigma=-1$ should be excluded. It follows from \eqref{ThImprovedReal5} that when $\sigma<0$ and $\sigma \neq -1$, i.e. $\sigma=\left|\sigma\right| e^{i\pi}$ and $\left|\sigma\right| \neq 1$ there is only one the nearest to the origin singularity given by \eqref{singtau1}
\begin{equation} \label{case:sigma<0}
\tau_\ast^{(1)}=1+\sigma-\sigma\ln \left| \sigma \right|
\end{equation}
that lies on the positive real axis.
Correspondingly the radius of convergence instead of \eqref{ThImprovedReal1} is
the modulus of the right-hand side of \eqref{case:sigma<0}.

Finally, when $\sigma=0$ the series following from \eqref{eq:basic}
\begin{equation}\label{case:sigma=0}
u=-\ln(1-\tau)=\sum_{m=1}^\infty \frac{\tau^m}{m}
\end{equation}
is convergent for $\left|\tau\right|<1$.

\textit{Note}. When $\sigma=-1$, $\tau_\ast^{(1)}=0$ by \eqref{case:sigma<0}, i.e. the series diverges everywhere.
We also note that in all cases considered above the condition of convergence of the series \eqref{eq:improvedStir}
is described in an uniform manner, particularly, the radius of convergence is given by \eqref{ThImprovedReal1}.

%--------------------------------------------------------------------------------------------------------------------------
\subsection{ Comparison with series \eqref{eq:Comtet} }
%--------------------------------------------------------------------------------------------------------------------------
Let us compare the domain of convergence for the series \eqref{eq:Comtet} and \eqref{eq:improvedStir}. Both can be represented in the form
\begin{equation}\label{powerstau}
u=\sum_{m=1}^\infty c_m(\sigma)\tau^m
\end{equation}
(see \eqref{doublesum} for the series \eqref{eq:Comtet}). However, by Corollary \ref{cor:5.1} and Corollary \ref{cor:x>e} the series \eqref{eq:improvedStir} has a much wider domain of convergence than the series \eqref{eq:Comtet} (not only in the real case but also in the complex case, see Figure \ref{Fig:complex} below). To undestand this phenomenon we note that
both series \eqref{eq:Comtet} and \eqref{eq:improvedStir} are defined for $x>1$, therefore the closer the boundary of a domain of convergence to 1, the wider domain of convergence.
However, the expansion coefficients $c_m(\sigma)$ in the series \eqref{eq:Comtet} are given by power series near $\sigma=0$ whereas in the series \eqref{eq:improvedStir}, the expansion coefficients are defined through function $\zeta=\zeta(\sigma)$, i.e. $c_m(\sigma)=c_m(\zeta(\sigma))$, where $\zeta(\sigma)=1/(1+\sigma)$, and are given by power series near $\zeta=0$.
Since $\sigma=1/\ln x$ becomes larger as $x$ is approaching 1, the series \eqref{eq:improvedStir}  has a wider domain of convergence. This corresponds to the fact that the function $\zeta=\zeta(\sigma)$ maps the interior of the unit circle $\left|\sigma\right|=1$ into an unbounded domain which is the right half-plane $\R\zeta>1/2$. We also note that the series \eqref{eq:Comtet} and \eqref{eq:improvedStir} have common values in the domain where they are both convergent, therefore the series \eqref{eq:improvedStir} is an analytic continuation of the series \eqref{eq:Comtet}.

In terms of variable $\zeta$ the series \eqref{eq:improvedStir} becomes \cite{CRpaper}
\begin{equation} \label{dzetaStir2}
u=\sum_{m=1}^\infty \frac{\tau^m}{m!} \sum_{p=0}^{m-1}
\StirSubSet{p+m-1}{p}(-1)^{p+m-1}\zeta^{p+m}
\end{equation}
and can be regarded as a result of applying the Euler's transformation for improvement of convergence of series \cite{Hardy}. Indeed, the standard Euler's transformation associated with changing variable to extend a domain of convergence of the series \eqref{eq:Comtet} is $\rho=\sigma/(1+\sigma)$ \cite{Morse}. Since in terms of a new varibale the fundamental relation (\ref{eq:basic}) is written as
$$
1-\rho=\frac{u}{e^{-u}+u+\tau-1},
$$
it would be natural to introduce variable $\zeta=1-\rho=1/(1+\sigma)$ rather than $\rho$. The series \eqref{eq:improvedStir},\eqref{dzetaStir2} were first found in \cite{CRpaper}.

One can also show that a representation of $W$ function through the function $u_\tau(\sigma)=\sum_{m=1}^\infty c_m(\tau)\sigma^m$, where $\tau$ plays a role of parameter, can not extend the domain of convergence established for series \eqref{eq:improvedStir}. Indeed, in this case equation \eqref{ThImprovedReal3} changes to \mbox{$F_\tau(\sigma,u)=0$} where $F_\tau(\sigma,u)$ is still defined by the right-hand side of \eqref{ThImprovedReal4} but with initial condition $u_\tau(0)=-\ln(1-\tau)$. By the Implicit Function Theorem it should be \mbox{$\partial F_\tau(0,-\ln(1-\tau))/\partial u$} $\neq 0$, which gives $\tau\neq 1$, i.e. $\left|\tau\right|<1$, and substituting $\tau=-\sigma\ln \sigma$ yields $0<\sigma<1/\omega_0$ as a necessary condition for convergence (cf. $0<\sigma<\sigma_1$ in Corollary \ref{cor:5.1}).

Thus among the series with the considered  structures the series \eqref{eq:improvedStir} has as wide as possible domain of convergence.

%-------------------------------------------------------------------------------------------------------------------------
\subsection{Asymptotics of expansion coefficients} \label{AsymptCoeff}
%-------------------------------------------------------------------------------------------------------------------------
Once the behavior of function $u=u_\sigma(\tau)$ near the nearest to the origin singularities has been established one can find an asymptotic formula for the expansion coefficients of the series \eqref{eq:improvedStir} using the Darboux's theorem about expansions at algebraic singularities \cite{Comtet,Bender}. The similar approach, based on the Weierstrass's preparation theorem and the Darboux's theorem, was applied to asymptotic enumeration of trees in \cite{SavickyWoods}.

According to the Darboux's theorem and found estimates \eqref{behsing} for $\sigma>0$ the expansion coefficients in the series \eqref{powerstau}
have an asymptotic formula for large $m$ as
$$
c_m(\sigma)=\frac{1}{m}
\left(
i\sqrt{\frac{2\tau_\ast^{(0)}}{\sigma}}\frac{1}{\Gamma\left(-\frac{1}{2}\right)\left( \tau_\ast^{(0)} \right)^m m^\frac{1}{2}} - i\sqrt{\frac{2\tau_\ast^{(1)}}{\sigma}}\frac{1}{\Gamma\left(-\frac{1}{2}\right)\left(\tau_\ast^{(1)} \right)^m m^\frac{1}{2}}
\right)+o\left(\frac{1}{m^\frac{3}{2}\left|\tau_\ast^{(1)}\right|^m}\right)
$$
or
\begin{equation}\label{asympt1}
c_m(\sigma)\sim \frac{i}{\sqrt{2\pi\sigma}m^\frac{3}{2}}
\left(
\frac{1}{\left( \tau_\ast^{(1)} \right)^{m-\frac{1}{2}}}-\frac{1}{\left( \tau_\ast^{(0)} \right)^{m-\frac{1}{2}}}
\right),
\end{equation}
as $\Gamma\left( -\frac{1}{2} \right) = -2\sqrt{\pi}$. Setting $\tau_\ast^{(1)}=\left|\tau_\ast^{(1)} \right|e^{i\theta_1}$ we find
\begin{equation}\label{asympt2}
c_m(\sigma)\sim \sqrt{\frac{2}{\pi\sigma}} \frac{\sin\left(m-\frac{1}{2}\right)\theta_1}
{\tau_\ast^{m-\frac12} m^\frac{3}{2}}, \hspace{5mm} \mbox{ as } \hspace{2mm} m\rightarrow\infty
\end{equation}
where $\tau_\ast=\tau_\ast(\sigma)$ is defined by \eqref{ThImprovedReal1} and $\theta_1=\arg(1 + \sigma -\sigma \ln \sigma + i\pi\sigma)$.\\
Specifically, for $\sigma>0$
$$
\theta_1 =
\left\{
\begin{array}{cc}
 \arctan\displaystyle\frac{\pi}{1-\ln\sigma+1/\sigma},& \hspace{3mm} \mbox{if\hspace{2mm}$0<\sigma< \displaystyle \frac{1}{W\left(1/e\right)}$}, \\[10pt]
 \pi+\arctan\displaystyle\frac{\pi}{1-\ln\sigma+1/\sigma},& \hspace{-5mm} \mbox{if\hspace{2mm}$\sigma> \displaystyle \frac{1}{W\left(1/e\right)}$}.
\end{array}
\right.
$$

It follows from \eqref{asympt2} that for large $m$ the expansion coefficients in the series \eqref{eq:improvedStir} disclose their oscillatory behavior  due to {\sl sin} function though the amplitude decays as $\tau_\ast(\sigma)>1$ for any $\sigma>0$. Since the series \eqref{eq:improvedStir} can be interpreted as a result of applying the Euler's transformation to the series \eqref{eq:Comtet} (cf. \eqref{dzetaStir2}), we note that some cases of oscillatory coefficients resulting from the Euler's transformation are studied in \cite{Hunter}.

In order to find an asymptotic formula in case when $\sigma<0$, suffice it to take in \eqref{asympt1} only the first term with \eqref{case:sigma<0}
\begin{equation}\label{asympt3}
c_m(\sigma) \sim -\frac{1}{\sqrt{2\pi\left|\sigma\right|}m^{\frac{3}{2}} \left( 1-\left|\sigma\right| + \left|\sigma\right| \ln \left|\sigma\right|       \right)^{m-\frac{1}{2}}}  \hspace{5mm} \mbox{ as } \hspace{2mm} m\rightarrow\infty
\end{equation}

Finally, for case $\sigma=0$, it follows from \eqref{case:sigma=0} that for any $m\in\N$
\begin{equation}\label{asympt4}
c_m(0)=\frac{1}{m}.
\end{equation}

%-----------------------------------------------------------------------------------------------------
\subsection{Convergence in complex case}
%-----------------------------------------------------------------------------------------------------
Theorem \ref{Th:ConvReal} is extended to the complex case.
\begin{Th} \label{Th:ConvComplex}
For complex $\sigma$, the radius of convergence of the series \eqref{eq:improvedStir} for $\alpha=1$ is
\begin{subequations} \label{eqs:ComplexConv}
\begin{equation}
\tau_\ast(\sigma)=\left|1+\sigma-\sigma\ln \sigma - i\pi\sigma\right| \mbox{ when } \Im\sigma<0 \ ,
\end{equation}
\begin{equation}
\tau_\ast(\sigma)=\left|1+\sigma-\sigma\ln \sigma + i\pi\sigma\right| \mbox{ when } \Im\sigma>0 \ .
\end{equation}
\end{subequations}
In the complex $z$-plane this is equiavalent to
the series \eqref{eq:improvedStir} for $W$ function is convergent everywhere in the exterior of the boundary line defined by equation
\begin{equation} \label{complexz}
\left|-\sigma\ln\sigma\right|=\left|1+\sigma-\sigma\ln \sigma\pm i\pi\sigma\right| \ ,
\end{equation}
where $\sigma=1/\ln z$ and sign minus or plus is taken respectively in the upper or lower half-plane.
\end{Th}

\begin{proof}
Repeating the proof of Theorem~\ref{Th:ConvReal} under assumption $\sigma\in\C$
we obtain the same equations \eqref{u:crit} and \eqref{ThImprovedReal5}
for singular points $u_\ast^{(k)}$ and $\tau_\ast^{(k)}$ respectively, where $k\in\Z$.
However, many of the singular points do not correspond to the principal branch of $W$ function
and relate to the other branches. We are going to find acceptable values of $k$
for which singular points relate to the principal branch of $W$. 
%Formally these values of $k$ are not obvious
%because the logarithms (and their branches) in the equations \eqref{eq:GenAsymp} and \eqref{u:crit}
%are different, more precisely, they are taken of different variables,
%$z$ in the former and $\sigma=1/\ln z$ in the latter.

To find acceptable values of $k$,
we utilize a relation
$\tau=-\sigma\ln\sigma$ following from \eqref{eq:st} to obtain
\begin{equation} \label{u:sigma}
u=W(e^s)-1/\sigma-\ln\sigma \ ,
\end{equation}
where $s=\frac{1-\tau}{\sigma}-\ln \sigma$.
Let us consider values of $u$ in the $\epsilon$-vicinity of the point $u_\ast^{(k)}$.
Comparing between \eqref{u:sigma} and \eqref{u:crit} gives
\begin{equation}
i\pi(2k-1)+\epsilon e^{i\varphi}=W(e^s)-1/\sigma \ ,     \notag
\end{equation}
where $-\pi<\varphi\leq\pi$. Setting $z=\left|z\right|e^{i\theta}$ ($-\pi<\theta<\pi$) in
$\sigma=1/\ln z$ and separating the imaginary part in the last equation we obtain
\begin{equation} \label{eq:k}
\pi(2k-1)+\epsilon\sin\varphi=\Im W -\theta \ .
\end{equation}
Since for the principal branch $-\pi<\Im W <\pi$, we find $-1/2-\epsilon\sin\varphi/(2\pi)<k<3/2-\epsilon\sin\varphi/(2\pi)$,
i.e. acceptable values are $k=0$ and $k=1$.

Now we note that both points $\tau_\ast^{(0)}$ and $\tau_\ast^{(1)}$
are singular, particularly, they correspond to a square-root branch point
of function $u=u_\sigma(\tau)$ for the same reason as in the real case (see proof of Theorem \ref{Th:ConvReal}).
Taking into account this result
we consider equation \eqref{eq:k} for $\epsilon=0$ in two cases $k=0$ and $k=1$.
When $k=0$, we have $\Im W=\theta-\pi$. Since $-\pi<\Im W <\pi$,
only positive $\theta$ satisfy this equation, i.e. $0<\theta<\pi$.
Similarly when $k=1$ we have $\Im W=\theta+\pi$ which holds for $-\pi<\theta<0$.
Thus we conclude that the curve $\left|-\sigma\ln\sigma\right|=\left|\tau_\ast^{(0)}(\sigma)\right|$ is located in the upper $z$-half-plane
and the curve $\left|-\sigma\ln\sigma\right|=\left|\tau_\ast^{(1)}(\sigma)\right|$ is located in the lower $z$-half-plane,
these curves being symmetric with respect to the real axis.
Hence, the equation \eqref{complexz} describes the boundary of domain of convergence
of the series \eqref{eq:improvedStir} in the complex case.
In addition, since $\sigma=(\ln\left|z\right|-i\theta)/\left|\ln z\right|^2$,
$\theta$ and $\Im\sigma$ are of opposite signs and the equations \eqref{eqs:ComplexConv} follow.
The theorem is completely proved.
\end{proof}

The curve defined by equation \eqref{complexz} in the complex $z$-plane is depicted in Figure \ref{Fig:complex} by solid line in the upper half-plane only (corresponding to the negative sign) as it is symmetric with respect to the real axis. The exterior of this boundary line  can be regarded as the domain of analytic continuation of the series \eqref{eq:improvedStir} from the part of the real axis $x>x_1$ (see Corollary \ref{cor:5.1}) to the complex $z$-plane. For comparison in the same figure it is shown (by dashed line) the boundary line of the domain of convergence of the series \eqref{eq:Comtet} defined by equation \eqref{eq:ConvComtetComplex}.

\begin{figure}[htb]
  \noindent\centering
  {
  \includegraphics[width=80mm]{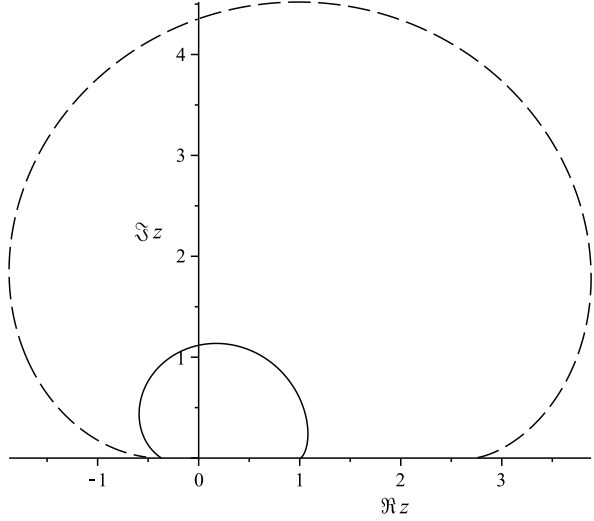}
  }
  \caption{The domains of convergence of series \eqref{eq:improvedStir} and \eqref{eq:Comtet} in the complex $z$-plane
           are located in the exterior of the curves depicted by solid and dashed lines respectively.}
  \label{Fig:complex}
\end{figure}

\begin{remark}\label{Remark6.1}
We note that the case $\left|\sigma\right|<1$ reveals a connection between
the series \eqref{eq:improvedStir} and \eqref{eq:Comtet}. In particular, the case permits
to expand $1/(1+\sigma)$ in powers of $\sigma$ in the former that after some rearrangments
can be reduced to the latter \cite{CRpaper}.
In accordance with Theorem~\ref{Th:ConvComtetComplex} the series \eqref{eq:Comtet} is convergent
in domain $V$ in the complex $\sigma$-plane defined by \eqref{eq:ConvComtetComplex} (written in terms of $\sigma=1/\ln z$).
One can show that the domain $V$ is contained in the unit disc $U=\left\{\sigma\in\C \: | \left|\sigma\right|<1\right\}$ (cf. Remark \ref{Remark4.2}) with the boundaries of $V$ and $U$ having one common point $\sigma=1$ (where both series are convergent).
The series \eqref{eq:improvedStir} is also convergent in $V$ but has a wider domain of convergence
being convergent (for $\left|\sigma\right|<1$) in $U\cap H$ where a domain $H$ bounded by curve \eqref{complexz}.
\end{remark}

In the end of this subsection we give asymptotics for the expansion coefficients of the series \eqref{eq:improvedStir} as $m\rightarrow\infty$ when $\Im\sigma\neq0$.  It follows from the proof of Theorem \ref{Th:ConvComplex} that in this case
there is only one singularity $\tau=\tau_\ast^{(0)}$ when $\Im\sigma<0$
and $\tau=\tau_\ast^{(1)}$  when $\Im\sigma>0$.
Therefore, one can use formula \eqref{asympt1} keeping only one corresponding term (unlike case of real $\sigma$ when there occur two singularities and both terms constitute the asymptotic formula \eqref{asympt2}). Thus, taking \eqref{singtau} we find
$$
c_m(\sigma) \sim \displaystyle \frac{1}{\sqrt{2\pi\sigma}m^{3/2}} \frac{\pm i}{\left( 1 + \sigma -\sigma \ln \sigma \pm i\pi\sigma \right)^{m-1/2}} \hspace{3mm}\mbox{as}\hspace{3mm}m\rightarrow\infty,
$$
where sign ''$+$'' (''$-$'') is taken in case of positive (negative) $\Im\sigma$.

%------------------------------------------------------------------------------------------------------------------------------
\subsection{Representation in terms of Eulerian numbers}
%------------------------------------------------------------------------------------------------------------------------------
The expansion coefficients of the series \eqref{eq:improvedStir} can be expressed in terms of the second-order Eulerian numbers \cite{{Graham et. al},{Corless1997seqseries}}. To show that we combine 
\begin{equation} \label{eq:u(sigma,tau)}
u=W(e^s)-\frac{1-\tau}{\sigma} \ ,
\end{equation}
where
\begin{equation} \label{eq:s(sigma,tau)}
s=s(\sigma,\tau)=\frac{1-\tau}{\sigma}-\ln \sigma \ .
\end{equation}
and \eqref{powerstau}, then the coefficients $c_m(\sigma)$ in the right-hand side of \eqref{powerstau}
are
\begin{equation} \label{eq:cmGE2}
c_m(\sigma)=\displaystyle \frac{1}{m!} \left. \displaystyle \frac{d^m}{d\tau^m} W \left( e^s \right)\right|_{\tau=0}=\displaystyle \frac{1}{m!} \left( -\displaystyle\frac{1}{\sigma} \right)^m \left. \displaystyle \frac{d^m}{ds^m} W \left( e^s \right)\right|_{s=-\ln \sigma + 1/ \sigma}
\end{equation}
as $\partial s/\partial\tau = -1/\sigma$ by \eqref{eq:s(sigma,tau)}.\\
Because of \eqref{eq:u(sigma,tau)} the formula \eqref{eq:cmGE2} is valid for $m\geq2$, for $m=1$ we have
\begin{equation} \label{cm1}
c_1(\sigma)=\frac{1}{\sigma}+\displaystyle \left. \displaystyle \frac{d}{d\tau} W \left( e^s \right)\right|_{\tau=0}.
\end{equation}
Since \cite{Corless1996Lam,Corless1997seqseries}
$$
\frac{d^m}{ds^m}W\left(e^s\right)=\frac{q_m\left(W\left(e^s\right)\right)}{\left(1+W\left(e^s\right)\right)^{2m-1}},
$$
where the polynomials $q_n(r)$ can be expressed in terms of the second-order Eulerian numbers \cite{Graham et. al,Corless1997seqseries}
$$
q_m(r)=\sum_{k=0}^{m-1}
\EulerianTwo{m-1}{k}(-1)^k r^{k+1},
$$
and
$$
\left. W \left( e^s \right)\right|_{s=-\ln \sigma + 1/ \sigma}=W\left( \displaystyle\frac{1}{\sigma}e^{1/\sigma} \right)=\displaystyle\frac{1}{\sigma},
$$
we finally obtain
\begin{equation} \label{CoeffImprovedEuler}
c_1(\sigma)=\frac{1}{1+\sigma} \mbox{,  } c_m(\sigma)= \displaystyle\frac{ (-1)^m\sigma^{m-1} } {m!(1+\sigma)^{2m-1} }
\sum_{k=0}^{m-1}\EulerianTwo{m-1}{k}\frac{(-1)^k} { \sigma^{k+1} } \mbox{,  } m\geq2.
\end{equation}
Substituting \eqref{CoeffImprovedEuler} into the right-hand side of \eqref{powerstau} results in a desireable formula

\begin{equation}\label{eq:improvedEuler}
 u = \frac{\tau}{1+\sigma} + \sum_{m=2}^\infty \frac{\tau^m}{m!(1+\sigma)^{2m-1}} \sum_{k=0}^{m-1}
\EulerianTwo{m-1}{k} (-1)^{m-k}\sigma^{m-k-2}.
\end{equation}
By introducing the variable $\zeta=1/(1+\sigma)$ the series \eqref{eq:improvedEuler} can also be written as

\begin{equation} \label{dzetaEuler2}
 u = \tau\zeta + \sum_{m=2}^\infty \frac{\tau^m}{m!}\sum_{k=0}^{m-1}
\EulerianTwo{m-1}{k}(-1)^{m+k}\zeta^{m+k+1} (1-\zeta)^{m-k-2}.         % \eqno (28)
\end{equation}
We note that the expansion \eqref{dzetaEuler2} does not contain terms of the second order in $\zeta$.

The series \eqref{eq:improvedStir},\eqref{dzetaStir2},\eqref{eq:improvedEuler}, and \eqref{dzetaEuler2} have the same properties including the domain of convergence and the asymtotic estimates for the expansion coefficients studied in Section \ref{AsymptCoeff}. This fact leads to some combinatorial consequences considered in Section \ref{C-R}.

%-----------------------------------------------------------------------------------------------------------------------------
\section{Transformed series}
%-----------------------------------------------------------------------------------------------------------------------------
\subsection{An invariant transformation} \label{subsec:InvTransform}
%-------------------------------------------------------------------------------------------------------------------------
The above studied expansions \eqref{eq:Comtet} and \eqref{eq:improvedStir} are limited in
their domain of applicability by the fact that $\sigma$ and $\tau$ are each singular at $z=1$, restricting their utility
to $z>1$. In addition to the domain of validity of the variables, there is the question of the domain of convergence
of the series ascertained in theorems \ref{Th:ConvComtetReal}, \ref{Th:ConvComtetComplex}, \ref{Th:ConvReal} and \ref{Th:ConvComplex}.

In this section we consider transformations of the series \eqref{eq:Comtet} and \eqref{eq:improvedStir}
referred as to transformed series. Our aims are to improve the convergence properties with respect to domain of convergence
and rate of convergence. Some results are obtained in \cite{CASC2010} employing experimental approach;
we supplement them with a theoretical study of convergence of the transformed series.

We reconsider the derivation of \eqref{eq:Comtet},
trying the \emph{ansatz}
\begin{equation}\label{eq:W(z,p)}
W = \ln z - \ln(p+\ln z)+u\ .
\end{equation}
Substituting into the defining equation $W e^W=z$, we obtain
\[  \bigg(\ln z -\ln(p+\ln z) +u\bigg)\frac{ze^u}{p+\ln z} = z
\]
From this, it is clear that if we define
\begin{equation} \label{def:sigma,tau(z,p)}
\sigma = \frac1{p+\ln z} \mbox{  and  } \tau = \frac{p+\ln(p+\ln z)}{p+\ln z}\ ,
\end{equation}
then we recover the equation \eqref{eq:basic} originally given by de Bruijn for $u$
and leading to the series \eqref{eq:Comtet}. Thus the fundamental relation (\ref{eq:basic}) is invariant with respect to $p$, with only the definitions of $\sigma$ and $\tau$ being changed. This remarkable property is due to a similarity property of the solution \eqref{eq:opening} of the original transcendental equation $y^\alpha e^y=x$ and explained by the following theorem.
\begin{Th}\label{th:similarity}
An untransformed series solution of equation $y^\alpha e^y=x$ and
the corresponding transformed series for the Lambert $W$ function, defined by \eqref{def:sigma,tau(z,p)},
are connected by relations
\begin{equation} \label{eq:x=(alpha*z)^alpha}
x=(\alpha z)^\alpha, \quad \alpha=e^p \ .
\end{equation}
\end{Th}
\begin{proof}
The solution \eqref{eq:opening} of equation $y^\alpha e^y=x$ possesses a similarity property
with respect to parameter $\alpha>0$ in the sense \cite{CRpaper}
\begin{equation} \label{eq:similarity}
\Phi_\alpha(x)= \alpha\Phi_1 \left( \frac{x^{1/\alpha}}{\alpha} \right)
= \alpha W \left( \frac{x^{1/\alpha}}{\alpha} \right)\ .
\end{equation}
It follows from \eqref{eq:opening} and \eqref{eq:similarity} that
\begin{equation} \label{eq:structure}
W \left( \frac{x^{1/\alpha}}{\alpha} \right) = \frac{1-\tau}{\sigma} + u \ ,
\end{equation}
where $\sigma$ and $\tau$ are defined by \eqref{eq:st}.
The right-hand side of \eqref{eq:structure} does not include $\alpha$ explicitly.
On the other hand, $\alpha$ is included in the left-hand side through a combination $z = x^{1/\alpha}/ \alpha$.
Therefore, the fundamental relation \eqref{eq:basic} will retain if we change variable $x=(\alpha z)^\alpha$.
Substituting this formula into \eqref{eq:st} and introducing parameter $p=\ln \alpha$
we obtain exactly equations \eqref{def:sigma,tau(z,p)} and the theorem follows.
\end{proof}
Thus introducing the invariant parameter $p$ generates an infinite one-parameter family of series formed by replacement of variables $\tau$ and $\sigma$ in the original series by expressions \eqref{def:sigma,tau(z,p)}.
Similar series for $W$ are associated with an invariance observed in \cite{CRpaper}
and studied in \cite{Corless1997seqseries}.

We now consider the properties of the transformations \eqref{def:sigma,tau(z,p)} for $z\in\R$. We shall start with $p\in \R$ and
later consider briefly one complex value of $p$.
Both $\sigma$ and $\tau$ are singular at $z_s=e^{-p}$, with the special case $p=0$
recovering the previous observations regarding the singularities at $z=1$.
We note $\sigma$ is monotonically decreasing
on $z>z_s$.
For $\tau$, we have $\tau(z_0)=0$ at $z_0=\exp(z_s-p)$, with $\tau$ positive for larger $z$ and negative for smaller.
Also we note that $\tau$ has a maximum at $z=\exp(ez_s-p)$.
In Figure \ref{Sigma}, we plot $\sigma$ and $\tau$, defined by (\ref{def:sigma,tau(z,p)}), for different values of $p$.
We see that for all $z>z_s$,
$\sigma$ decreases with increasing $p$, but $\tau$ increases.
In view of the form of the double sums above it is not obvious
whether convergence is increased or decreased as a result of these opposed changes.
This is what we wish to investigate here.

\begin{figure}[ht]
 \begin{center}
  \scalebox{0.33}{\includegraphics{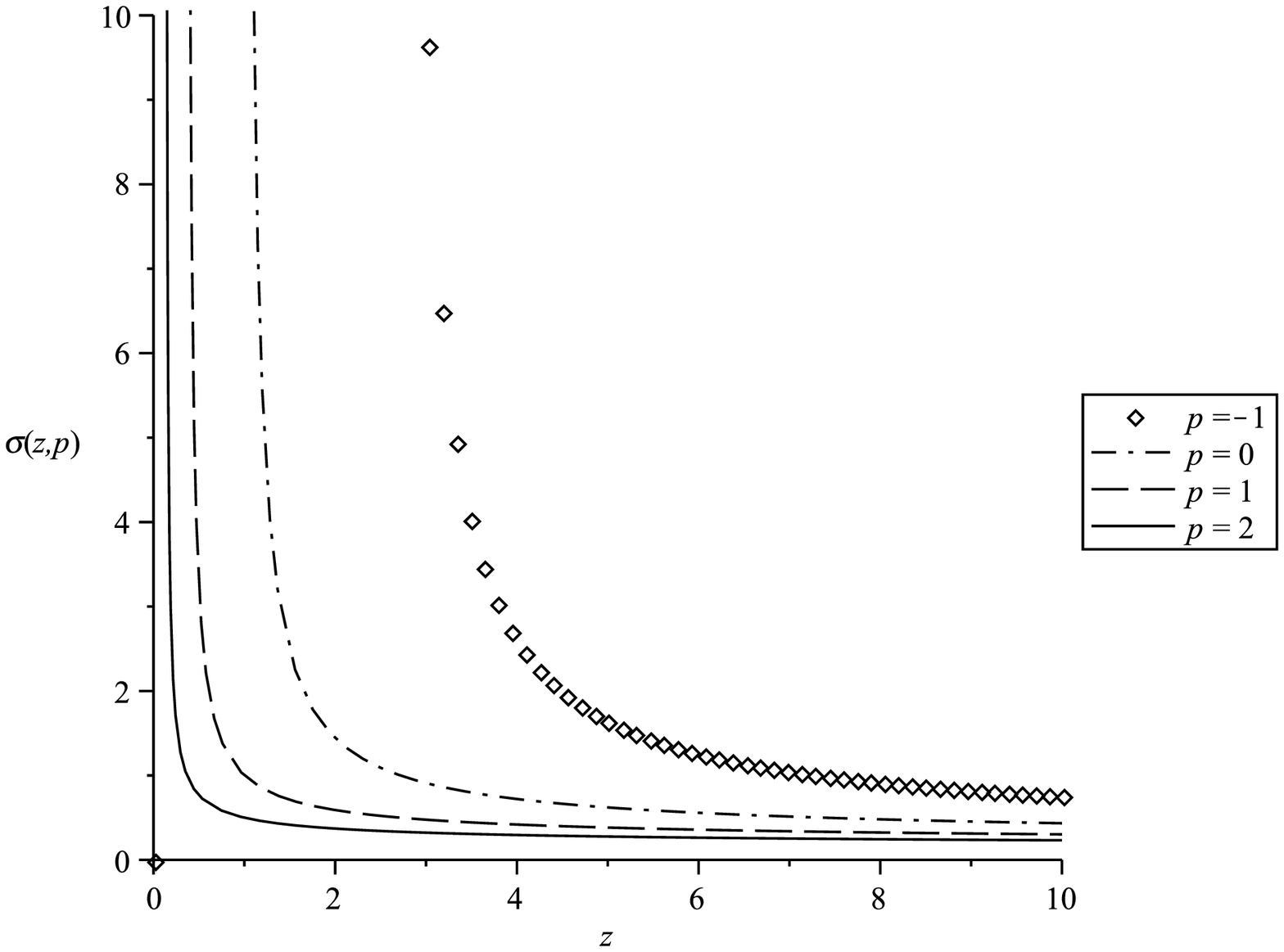}}
  \scalebox{0.3}{\includegraphics{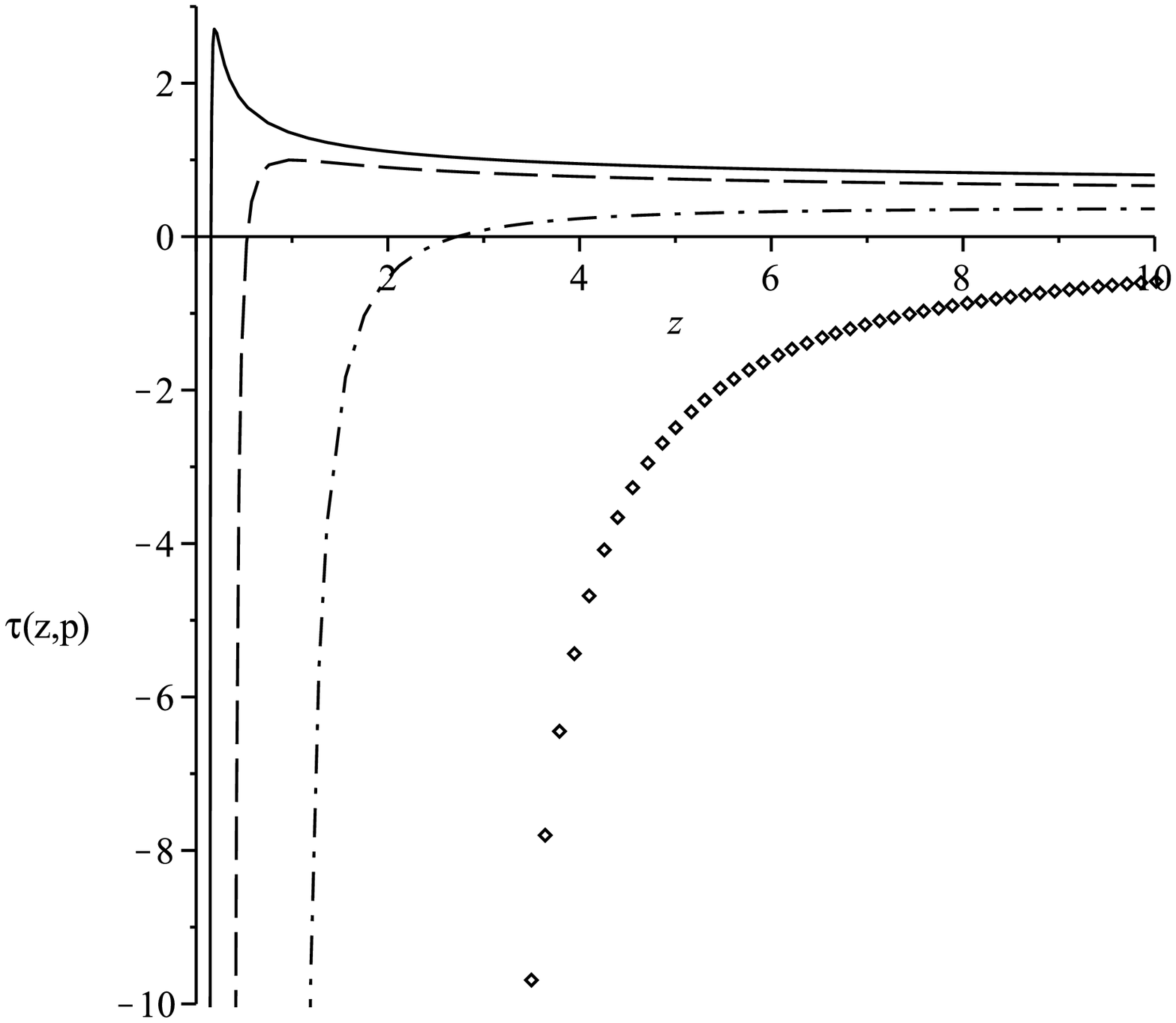}}
 \end{center}
 \caption{Dependence $\sigma$ and $\tau$ on $z$ for different values of parameter $p$.}
 \label{Sigma}
\end{figure}

%-----------------------------------------------------------------------------------------------------------------------
\subsection{Domain of convergence}
%-----------------------------------------------------------------------------------------------------------------------
We wish to investigate first the domains of $z\in\R$ for which the series \eqref{eq:Comtet} and \eqref{eq:improvedStir} converge,
and how the domains vary with $p$. We begin with theoretical results.

For $p=0$ the domains of convergence are known from theorems \ref{Th:ConvComtetReal} and \ref{Th:ConvReal}. Specifically, the series \eqref{eq:Comtet} converges for $z>e$ and the series \eqref{eq:improvedStir} converges for $z>z_0=1.004458...$ (see Corollary \ref{cor:5.1}).
For arbitrary real $p$ the following statement can be proved for the series \eqref{eq:Comtet}.

\begin{Th} \label{Th:ComtetParam}
The domain of convergence of the transformed series \eqref{eq:Comtet} is defined by equations
\begin{equation} \label{ineq:param}
\Re W_{-1}\left(-e^{p-1}(p+\ln z)\right)>p-1 \mbox{ and } z>e^{-p} \ ,
\end{equation}
which is equivalent to
\begin{equation} \label{eq:ComtetParam}
z > z_p =
\begin{cases}
e^{1-2p}, & p \leq 0 \\
e^{-p+\eta_0\csc\eta_0}, & p>0
\end{cases}
\end{equation}
where $\eta_0\in(0,\pi)$ is the root of equation $\eta_0\cot\eta_0=1-p$.
\end{Th}

\begin{proof}
The proof of the theorem is similar to that of Theorem \ref{Th:ConvComtetReal} and based on an application
of Theorem~\ref{Th:ConvComtetSigmaTau} to the transformed series \eqref{eq:Comtet}.
In particular, substituting the expressions \eqref{def:sigma,tau(z,p)}
in \eqref{eq:ConvComtetSigmaTau} we obtain in the real case,
i.e. under assumption $p+\ln z>0$, the inequality \eqref{ineq:param}.
Applying Lemma \ref{lemma} to the latter we get \eqref{eq:ComtetParam},
where $z_p>e^{-p}$, which justifies the above assumption and the theorem follows.\end{proof}
\begin{remark}
The formula \eqref{eq:ComtetParam} can be obtained by substituting equations \eqref{eq:x=(alpha*z)^alpha} in \eqref{eq:ConvComtetReal}. However, in contrast to $x_\alpha$ defined by \eqref{eq:ConvComtetReal}, $z_p$
is monotone decreasing with $p$. Thus, the larger $p$,
the wider domain of convergence of the transformed series \eqref{eq:Comtet} (cf. Remark \ref{rem:x(alpha)}).
\end{remark}
\begin{remark}
In the formula \eqref{eq:ComtetParam}, when $p>0$ but $p\neq1$,
we can also write $z_p=e^{-p+(1-p)\sec\eta_0}$.
\end{remark}
\begin{remark}\label{Remark 7.2}
The convergence condition \eqref{ineq:param} can be extended to the case of complex $z$
similar to the extension of the condition \eqref{eq:Remark2.2}
for the untransformed series \eqref{eq:Comtet} by Theorem~\ref{Th:ConvComtetComplex}.
\end{remark}
To find out the domain of convergence of the transformed series \eqref{eq:improvedStir} one should substitute \eqref{def:sigma,tau(z,p)} in \eqref{ThImprovedReal1} and solve the obtained equation for $z$ as a function of $p$.
As a result, we obtain formulae to compute $z_p$ similar to those stated in Theorem \ref{th:AlphaDomains}.
We also can find a very good approximation for $z_p$ for $p<1$ using equation $x_\alpha\approx x_1$, stated in Remark \ref{rem:x=x1}, and Theorem \ref{th:similarity}.
Specifically, substituting \eqref{eq:x=(alpha*z)^alpha} in this equation we obtain
\begin{equation} \label{eq:z(p)approx}
z_p \approx e^{-p}\,(z_{\scriptscriptstyle 0})^{e^{-p}}\ .
\end{equation}

The accurate and approximate values of $z_p$ are depicted in Figure \ref{Fig:x(p)}
by solid line and circles respectively together with curve \eqref{eq:ComtetParam} (dashed line).
\begin{figure}
  \begin{center}
  \scalebox{0.45}{ \includegraphics{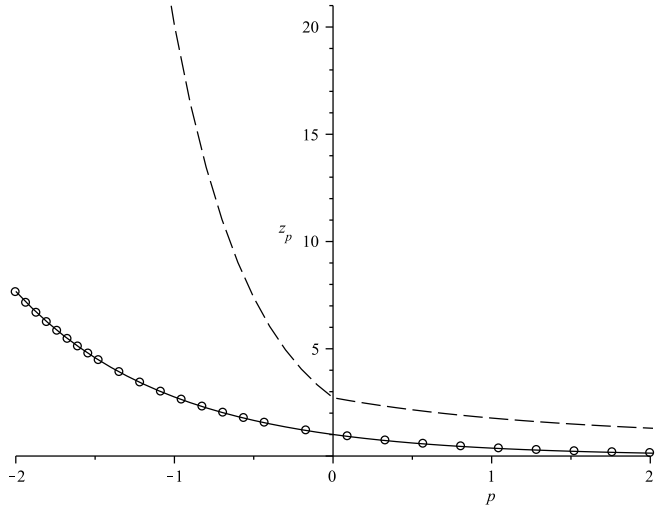} }
  \end{center}
  \caption{Behavior of boundary of convergence domain as a function of $p$ for series \eqref{eq:Comtet} (dashed line) and \eqref{eq:improvedStir} (solid line) in real case.}
  \label{Fig:x(p)}
\end{figure}
It follows from Figure \ref{Fig:x(p)} as well as from \eqref{eq:ComtetParam} and \eqref{eq:z(p)approx} that with increase of parameter $p$ the domain of convergence of the transformed series monotonely extends. To illustrate and qualitatively verify this result we design  an appropriate numerical procedure. The method is simply to compute the partial sum of a series to a high number of terms, using extended floating-point precision as necessary, and
then to plot the ratio of the partial sum to the exact value (the exact value is obtained using a built-in \textsc{Maple} function LambertW(k,x), where a method different from series summation is used).
The edge of the domain of convergence is then signaled by rapid oscillations and
by marked deviations from the desired ratio of $1$. (To make a graph be readable we depict only the relevant part of each curve.)

For the series (\ref{eq:Comtet}) we have plotted in Figure \ref{fig:ComtetPlot}
the partial sum to 40 terms for different values of $p$.
For $p=0$, we see a nice illustration of Theorem \ref{Th:ConvComtetReal}, with the partial sum becoming unstable in the vicinity of
$z=e$. For positive $p$, we see the domain of convergence increased and for negative $p$ it is decreased, in accordance with Theorem \ref{Th:ComtetParam}.
\begin{figure}[hbt]
  \begin{center}
  \scalebox{0.45}{\includegraphics{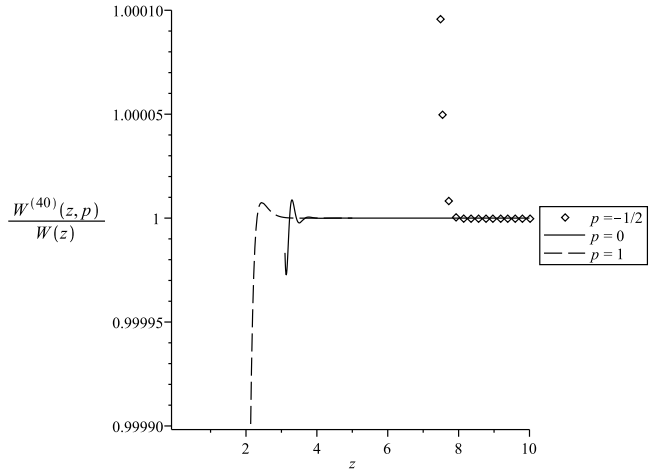}  }
  \end{center}
  \caption{For series (\ref{eq:Comtet}), the ratio $W^{(40)}(z,p)/W(z)$ as functions of $z$ for $p=-1/2,0,1$.}
  \label{fig:ComtetPlot}
\end{figure}
Similar effects can be seen for (\ref{eq:improvedStir}), we plot in Figure \ref{fig:allpimproved}
the partial sums for 40 terms as $p$ varies.
The domain of convergence for each $p$ is clearly seen, and confirms that the point of divergence
moves to larger $z$ for decreasing $p$ and to the left for increasing $p$. For $p=0$ this point is very close to 1, which sharp demonstrates the result in Theorem \ref{Th:ConvReal}.
\begin{figure}[hbt]
  \begin{center}
  \scalebox{0.4}{\includegraphics{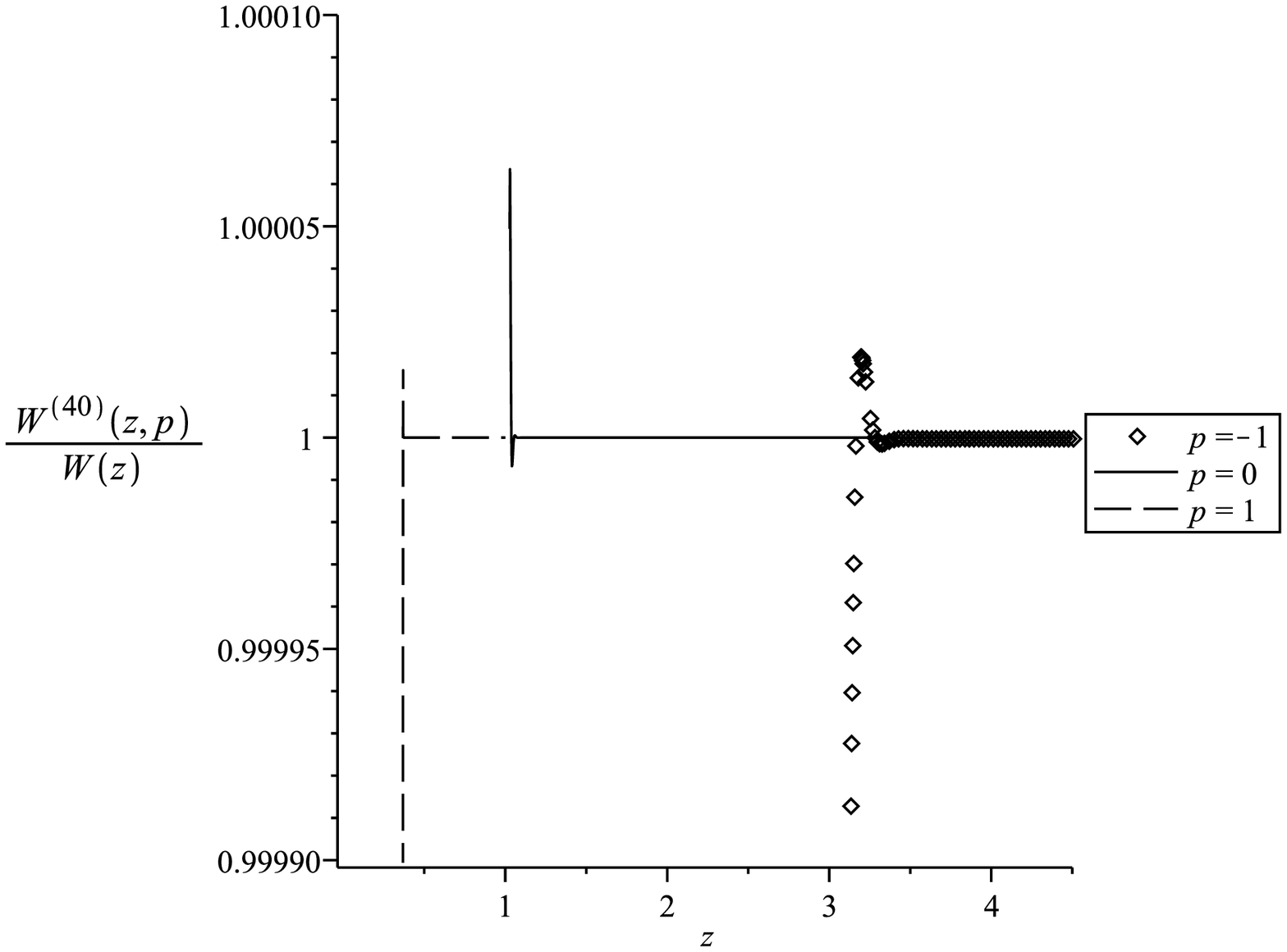}  }
  \end{center}
  \caption{For series (\ref{eq:improvedStir}), the ratio $W^{(40)}(z,p)/W(z)$ as functions of $z$ for $p=-1,0,1$.
  Compared with Figure \ref{fig:ComtetPlot}, this shows convergence down to smaller $z$.}
  \label{fig:allpimproved}
\end{figure}

We can summarize the above findings by noting that series (\ref{eq:improvedStir}) has a wider domain of convergence,
and a better behaviour with $p$, while the domain of convergence for series (\ref{eq:Comtet}) becomes worse in that order.

The fact that the domain of convergence of the transformed series is extending while the parameter $p$ is increasing can also be found in the complex case based on the results of theorems \ref{Th:ConvComplex} and \ref{Th:ComtetParam}. To make certain of this it is sufficient for the series \eqref{eq:improvedStir}, to substitute expressions \eqref{def:sigma,tau(z,p)} (with $z\in\C$) into equation \eqref{ThImprovedReal1} and for the series \eqref{eq:Comtet}, to consult Remark \ref{Remark 7.2}. The results are presented  for $p = -1, -1/2, 0 ,1/2 \mbox{  and  } 1$ in Figure \ref{fig:ComplexDomain2ap} and Figure \ref{fig:ComplexDomain3ap} for the series \eqref{eq:Comtet} and \eqref{eq:improvedStir} respectively where the curves for $p=0$ are the same as in Figure \ref{Fig:complex} and the points of intersection of the curves with the positive real axis correspond to points on the curves
depicted in Figure \ref{Fig:x(p)}.

\begin{figure}[ht]
  \begin{center}
  \scalebox{0.4}{ \includegraphics{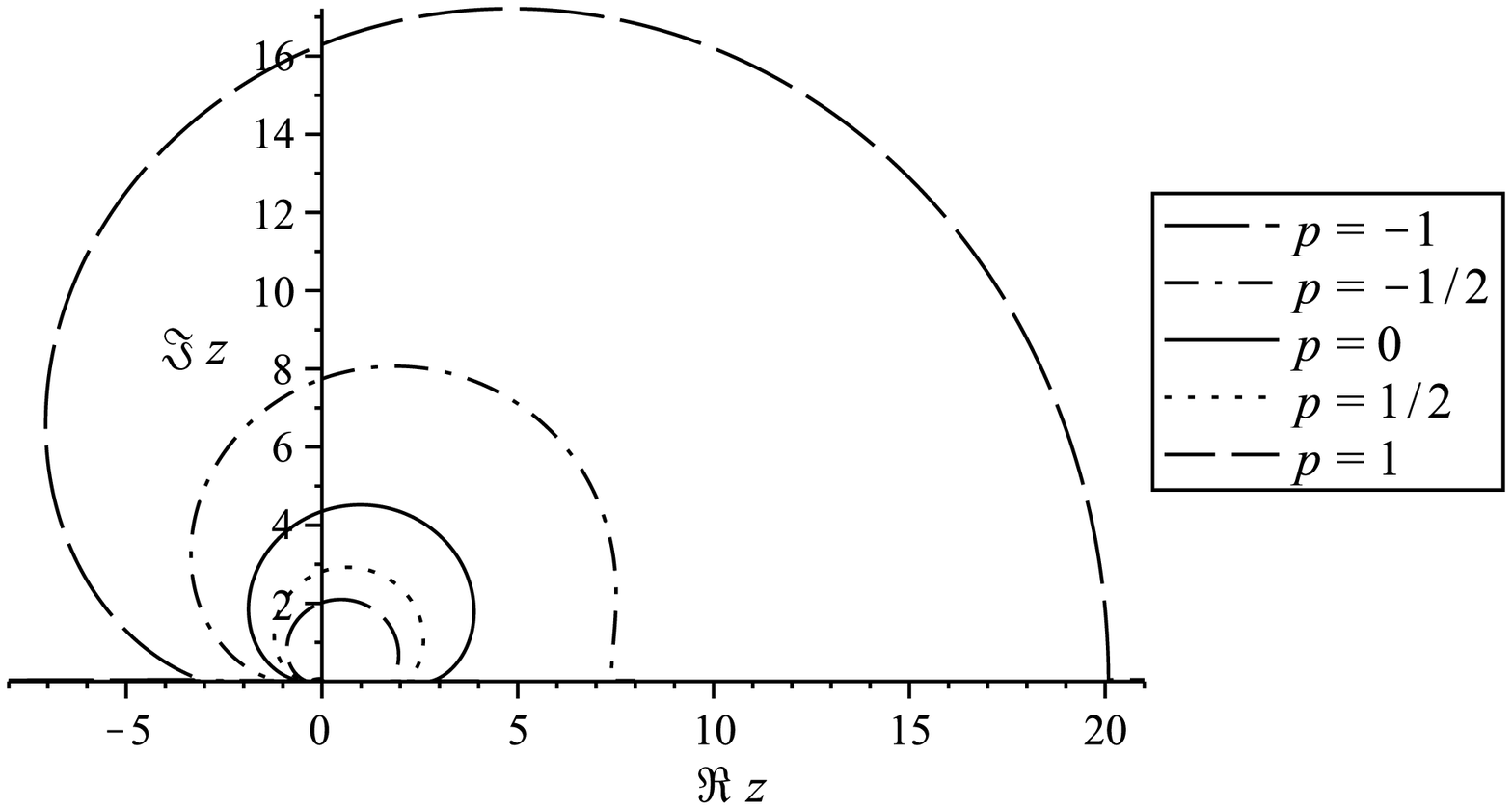} }
  \end{center}
  \caption{Domains of convergence of series (\ref{eq:Comtet}) in complex $z$-plane for $p = -1, -1/2, 0 ,1/2 \mbox{  and  } 1$.}
  \label{fig:ComplexDomain2ap}
\end{figure}

\begin{figure}[ht]
  \begin{center}
  \scalebox{0.4}{ \includegraphics{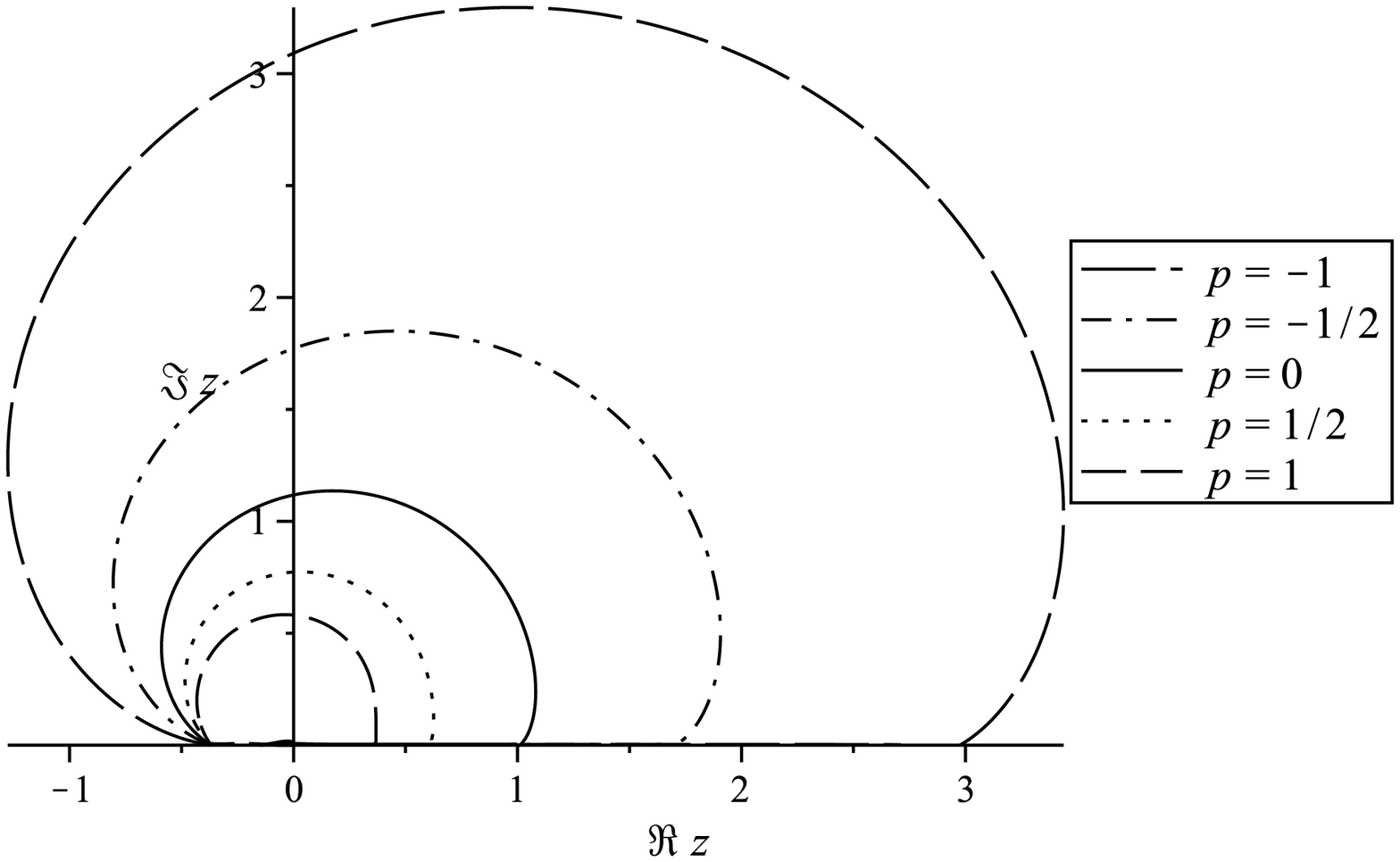} }
  \end{center}
  \caption{Domains of convergence of series (\ref{eq:improvedStir}) in complex $z$-plane for $p = -1, -1/2, 0 ,1/2 \mbox{  and  } 1$.}
  \label{fig:ComplexDomain3ap}
\end{figure}

%--------------------------------------------------------------------------------------------------------------------------
\subsection{Rate of convergence}
%--------------------------------------------------------------------------------------------------------------------------
By rate of convergence, we are referring to the accuracy obtained by partial sums of a series.
Given two series, each summed to $N$ terms, the series giving on average a closer approximation to the converged value
is said to converge more quickly.
The qualification `on average' is needed because it will be seen in the plots below that the error
regarded as a function of $z$ can show fine structure which confuses the search for a general trend.
Further, the comparison of rate of convergence between different series can vary with $z$ and $p$.
For some ranges of $z$, one series will be best, while for other ranges of $z$ a different series will be best.
Although one series may converge on a wider domain than another, there is no guarantee that the same series
will converge more quickly on the part of the domain they have in common.
The practical application of these series is to obtain rapid estimates for $W$ using
a small number of terms, and for this the quickest convergence is best, but
this will be dependent on the domain of $z$.

The previous section showed that positive values of the parameter $p$
extend the domain of convergence of the series, but its effect on rate of convergence is different.
Figures \ref{fig:ComtetAccuracy} and \ref{fig:ImproveAccuracy} show the
dependence on $z$ of the accuracy of computations of the series (\ref{eq:Comtet}) and (\ref{eq:improvedStir})
respectively with $N=10$ for $p=-1,-1/2,0$ and $1$.
One can see that the behaviour of the accuracy is non-monotone with respect to both $z$ and $p$ although some
particular conclusions can be made. For example, one can observe that for the series (\ref{eq:Comtet}) at least
for $z<30$ within the common domain of convergence the accuracy for $p=-1/2,0$ and $1$ is higher than for $p=-1$.
For the series (\ref{eq:improvedStir}) an increase of positive values of $p$ reduces a rate of convergence
within the common domain of convergence i.e. for $z>1.5$. However, at the same time for $z>11$ computations with $p=-1$ are more
accurate than those with positive $p$ and for $5<z<18$ the highest accuracy occurs when $p=-1/2$.
\begin{figure}
  \begin{center}
  \scalebox{0.4}{\includegraphics{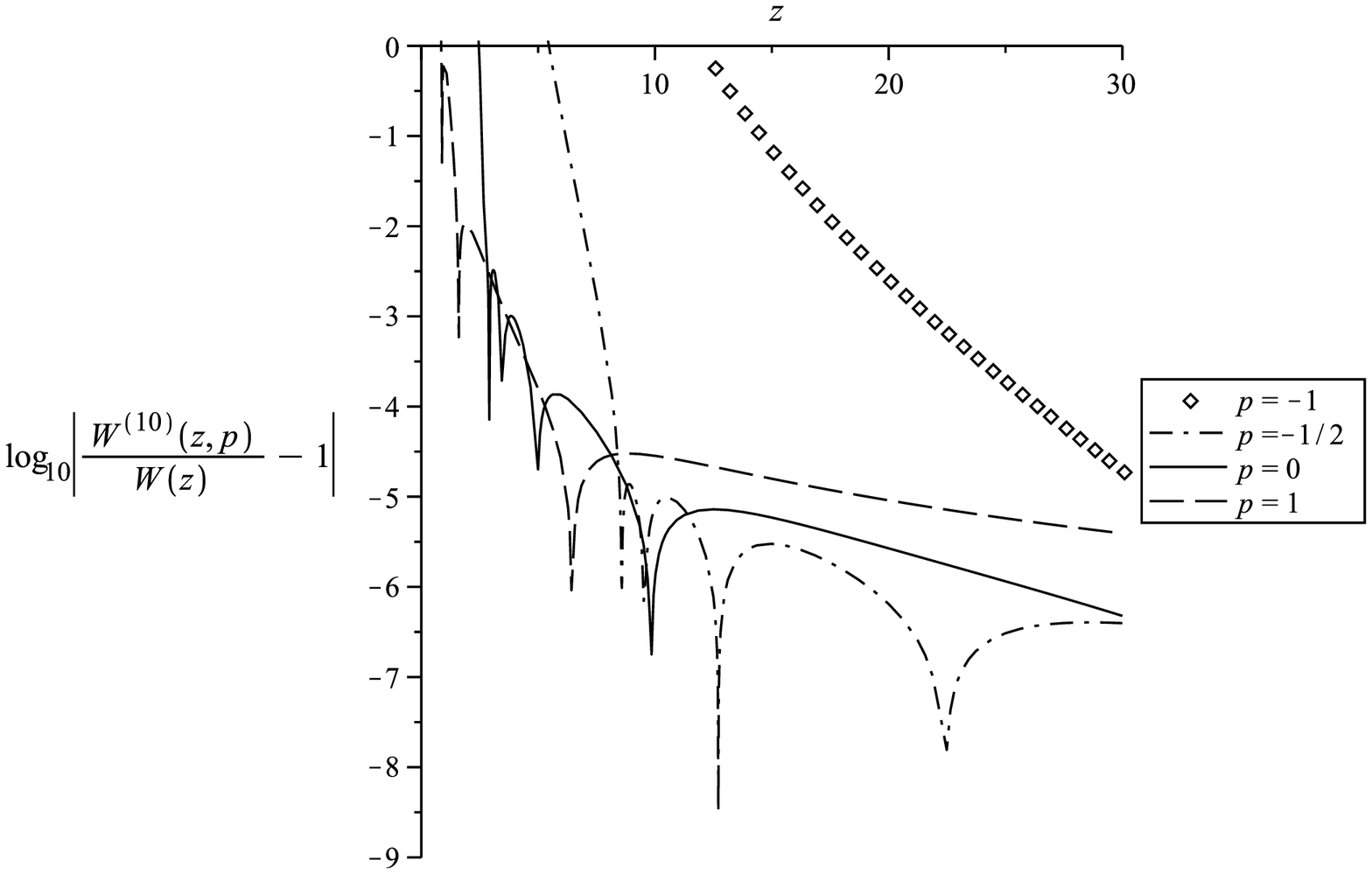} }
  \end{center}
  \caption{For series (\ref{eq:Comtet}) with $N=10$, changes in accuracy in $z$ for $p=-1,-1/2,0$ and $1$.}
  \label{fig:ComtetAccuracy}
 \end{figure}

 \begin{figure}
  \begin{center}
  \scalebox{0.4}{  \includegraphics{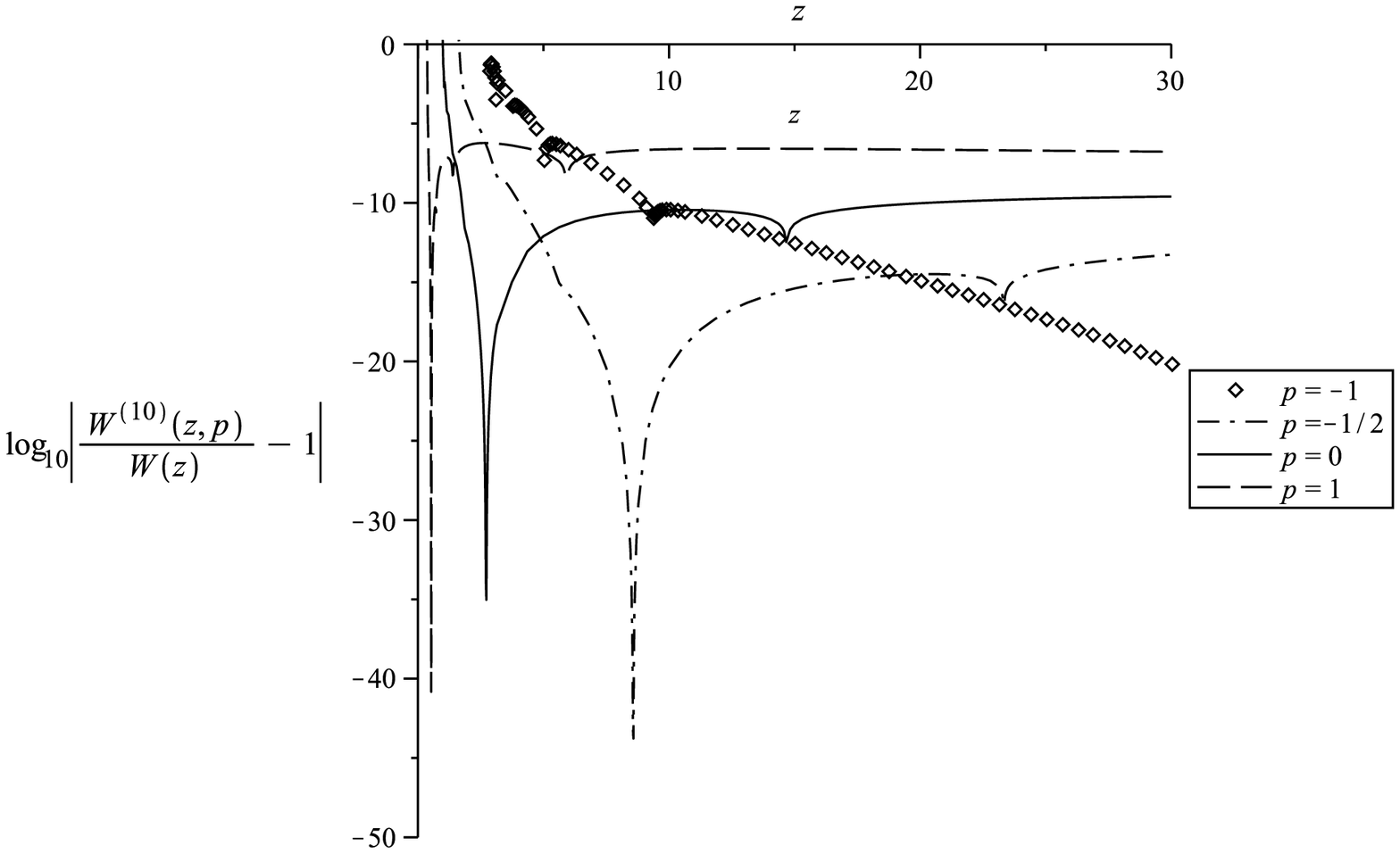}   }
  \end{center}
  \caption{For series (\ref{eq:improvedStir}) with $N=10$, changes in accuracy in $z$ for $p=-1,-1/2,0$ and $1$.}
  \label{fig:ImproveAccuracy}
 \end{figure}

The next two figures \ref{fig:Comtet Fixed z} and \ref{fig:Improved Fixed z}
display the dependence of convergence properties of the series (\ref{eq:Comtet}) and (\ref{eq:improvedStir}) respectively
on parameter $p$ for different numbers of terms $N=10,20$ and $40$.
Again, the curves in these figures confirm that the accuracy strongly depends on parameter $p$ and is non-monotone and
show that on the whole an increase of the number of terms improves the accuracy.
It is also interesting that there exists a value of $p$ for which the accuracy at the given point is maximum;
this value depends very slightly on $N$ and approximately is $p\approx-0.75$ in Figure \ref{fig:Comtet Fixed z}
and $p\approx-0.5$ in Figure \ref{fig:Improved Fixed z}.

 \begin{figure}
  \begin{center}
  \scalebox{0.4}{ \includegraphics{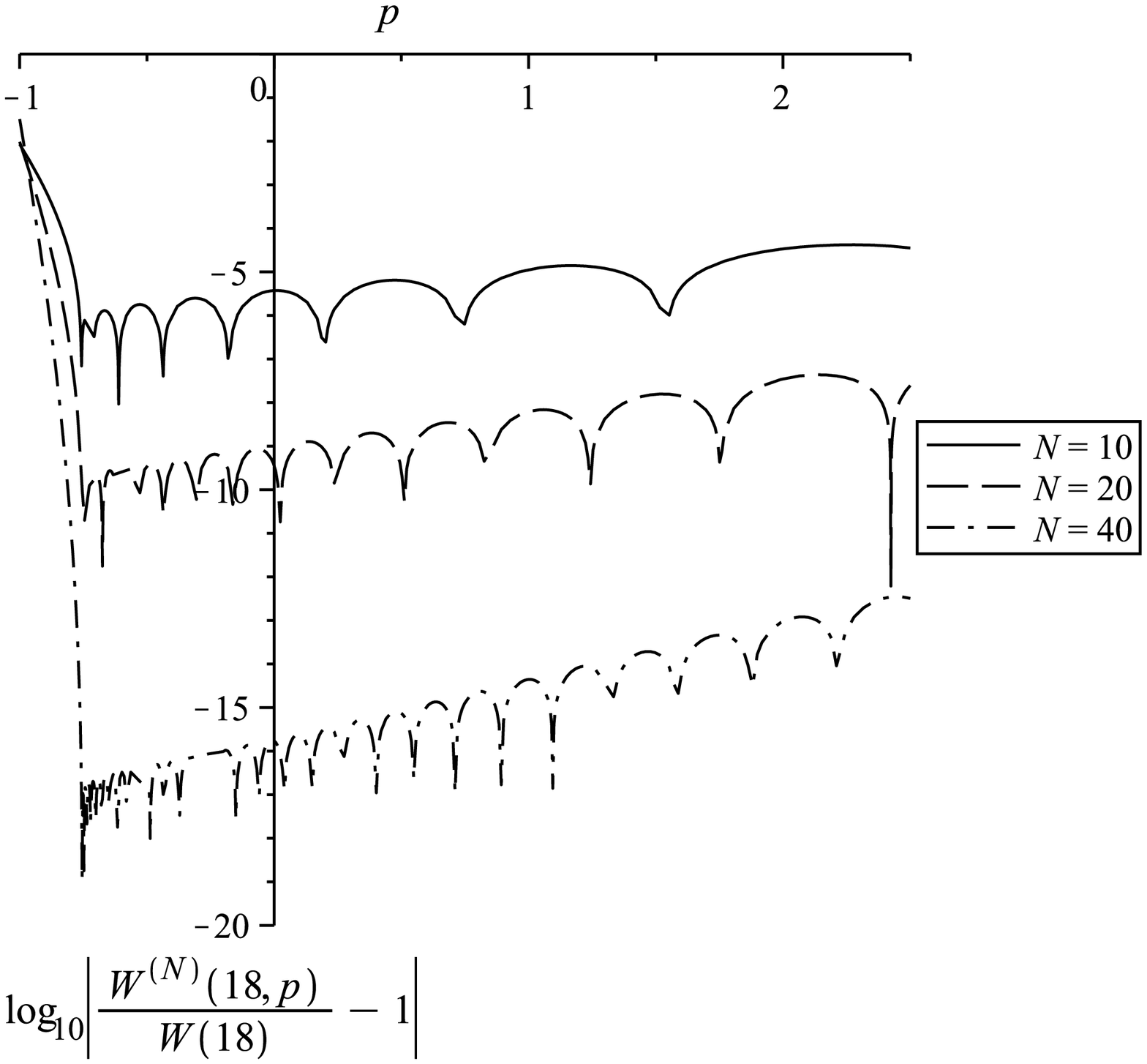}  }
  \end{center}
  \caption{For series (\ref{eq:Comtet}), the accuracy as a function of $p$ at fixed point $z=18$ for $N=10,20$ and $40$.}
  \label{fig:Comtet Fixed z}
 \end{figure}

 \begin{figure}
  \begin{center}
  \scalebox{0.4}{ \includegraphics{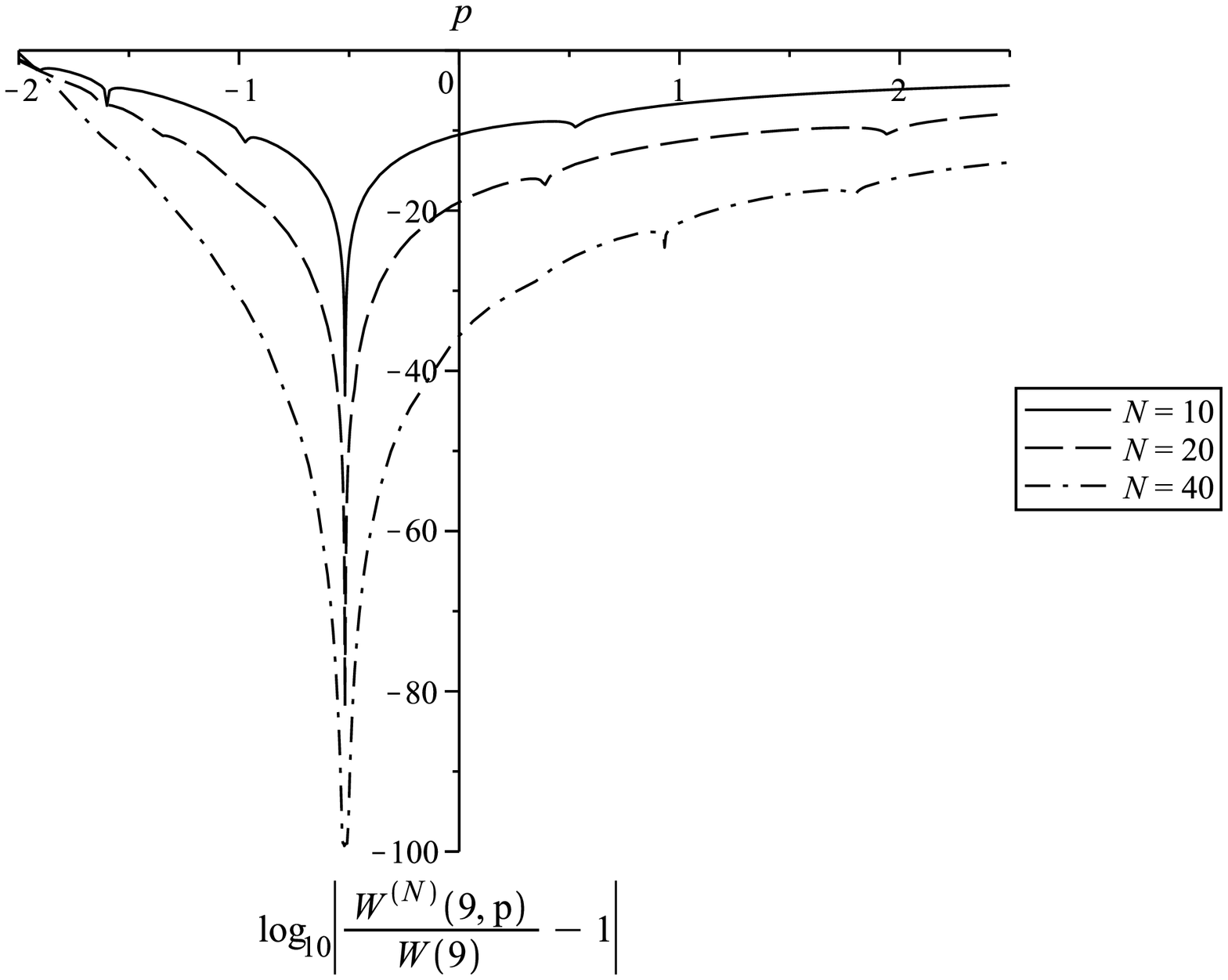}   }
  \end{center}
  \caption{For series (\ref{eq:improvedStir}), the accuracy as a function of $p$ at fixed point $z=9$ for $N=10,20$ and $40$.}
  \label{fig:Improved Fixed z}
 \end{figure}

The explained behaviour of the accuracy depending on parameter $p$ shows that introducing parameter $p$ in the series can
result in significant changes in accuracy.
The pointed out non-monotone effects of parameter $p$ on a rate of convergence can be due to the aforementioned
non-monotone behaviour of $\tau$.

%-------------------------------------------------------------------------------------------------------------------------------
\section{Branch $-1$ and complex $p$} \label{sec:minus1}
%-------------------------------------------------------------------------------------------------------------------------------
The above discussion has considered only real values for the parameter $p$.
We briefly shift our consideration to complex $p$ and to branch $-1$.
For $z$ in the domain $-1/e < z <0$ function $W_{-1}(z)$ takes real values
in the range $[-1,-\infty)$. The general asymptotic expansion takes the form \cite{Corless1996Lam}
\begin{equation}\label{eq:genminus1}
 W_{-1}(z) = \ln(z) - 2\pi i - \ln( \ln(z) - 2\pi i) + u \ ,
\end{equation}
where $u$ denotes the remaining terms similar to \eqref{eq:opening}.
This will clearly be very inefficient for $z\in [-1/e,0)$
because each term in the series will be complex, and yet the series must sum to a real number.
If, however, we utilize the parameter $p$, we can improve convergence enormously.

We again adopt the \emph{ansatz} used above to write
\begin{equation} \label{p:branch-1}
 W_{-1}(z) = [\ln_{-1} z+p] -[p+\ln (p+\ln_{-1} z)] +\frac{p+\ln(p+\ln_{-1} z)}{p+\ln_{-1} z}+v \ ,
\end{equation}
where $v$ stands for the remaining series which will not be pursued here.
By setting $p=i\pi$, we can rewrite $[\ln_{-1} z +i\pi]$ as $\ln(-z)$.
When $p=0$, \eqref{p:branch-1} is equivalent to \eqref{eq:genminus1}.
A numerical comparison of partial sums can be used to show the improvement.
Specifically, we compare the following first terms in \eqref{eq:genminus1} and
\eqref {p:branch-1}
\begin{eqnarray}
  W_{-1}^{(1)} &=& \ln(z) - 2\pi i - \ln( \ln(z) - 2\pi i) +\frac{\ln( \ln(z) - 2\pi i)}{\ln(z) - 2\pi i}
  \label{eq:oldm1} \ ,\\
  \hat W_{-1} &=& \ln(-z) - \ln(-\ln(-z)) + \frac{\ln(-\ln(-z))}{\ln(-z)}\ .
  \label{eq:betterm1}
\end{eqnarray}
The results are shown in Table \ref{tab:one} and graphically in Figure \ref{fig:Branchm1}.
One can see that both transformed ($p=i\pi$) and untransformed ($p=0$) series have an error that increases as $z\to -1/e$.
Although the maximum error occurs at $z=-1/e+0$, the transformed series is exactly correct at $z=-1/e$.
This is due to a difference in the local behaviour of $W_{-1}$ and approximation $\hat W_{-1}$ when $z\to -1/e +0$,
particularly, the former has a square-root singularity, while $\hat W_{-1}$ is regular there.
We also note that the transformed series is asymptotically correct as $z\to 0$.

\begin{table}
\begin{center}
\begin{tabular}{c|c|c|c}
$z$ & $W_{-1}(z)$ & $\hat W_{-1}(z)$ & $W_{-1}^{(1)}(z)$ \\   \hline
$-0.01$ &  $-6.4728$ & $-6.4640$ & $-6.3210-0.04815i$ \\
$-0.1$  & $-3.5772$ & $-3.4988$ & $-3.4124-0.3223i$ \\
$-0.2$ & $ -2.5426$ & $ -2.3810 $&$ -2.5182-0.5153i $ \\
$-0.3$ & $-1.7813$ & $ -1.5438 $&$ -2.0087-0.6621i $ \\
$-1/e$ & $-1$ &  $-1$& $ -1.7597-0.7450i$ \\
\end{tabular}
\end{center}
\caption{Numerical comparison of series transformation with $p=i\pi$. \label{tab:one}}
\end{table}
 \begin{figure}[ht]
  \begin{center}
  \scalebox{0.6}{ \includegraphics{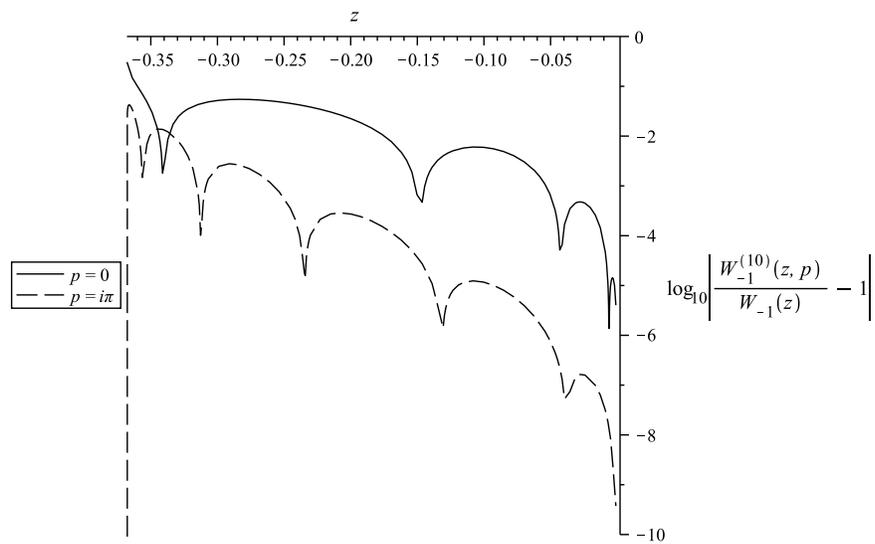}   }
  \end{center}
  \caption{Errors in approximations (\ref{eq:oldm1}) and (\ref{eq:betterm1}) for $W_{-1}$.}
  \label{fig:Branchm1}
 \end{figure}

\clearpage

%-----------------------------------------------------------------------------------------------------------------------
\section{Series \eqref{eq:wEuler}}
%-----------------------------------------------------------------------------------------------------------------------
\subsection{Different representations}
%--------------------------------------------------------------------------------
The series development \eqref{eq:wEuler} was obtained in \cite{Corless1997seqseries, Wright function}
and represents an expansion of $W(x)$ in powers of $\displaystyle\sigma^{-1}=\ln x$
\begin{equation} \label{Sum:lnx}
W(x)=\omega_0+\sum_{n=1}^\infty a_n(\ln x)^n
\end{equation}
or
\begin{equation} \label{Sum:t}
W(e^t)=\omega_0+\sum_{n=1}^\infty a_n t^n,
\end{equation}
where $t=\ln x$ and \cite{Wright function}
\begin{equation} \label{Coeff:Euler}
a_n=\displaystyle\frac{1}{n!(1+\omega_0)^{2n-1}}
\sum_{k=0}^{n-1}
\EulerianTwo{n-1}{k}(-1)^k \omega_0^{k+1}.
\end{equation}

The formula \eqref{Coeff:Euler} expresses the expansion coefficients $a_n$ in terms of the second-order Eulerian numbers \cite{Graham et. al,Corless1997seqseries}. We now show that these coefficients can also be represented through the unsigned associated Stirling numbers of the first kind $d(m,k)$ given by \cite{Comtet}
\begin{equation} \label{Def:d}
[\ln(1+v)-v]^k = k!\sum_{m=2k}^\infty (-1)^{m+k}\thinspace d(m,k)\frac{v^m}{m!}
\end{equation}
and the 2-associated Stirling numbers of the second kind used in the series \eqref{eq:improvedStir}.

Both representations can be obtained on the basis of a relation \cite{Unwinding}
\begin{equation}\label{OmegaRelation}
W(e^t)+\ln W(e^t) = t
\end{equation}
and the Lagrange Inversion Theorem \cite{Lagrange Th}.
To apply this theorem it is convenient to introduce a function that is zero at $t=0$. We consider function
\begin{equation}\label{Def:v}
v=v(t)=W(e^t)/\omega_0-1
\end{equation}
and write \eqref{OmegaRelation} as
$$
t=\omega_0\thinspace v+\ln (1+v).
$$
Then by the Lagrange Inversion Theorem we obtain
\begin{equation}\label{Sum:v}
v=\sum_{n=1}^\infty \frac{t^n}{n}[v^{n-1}]\left( \omega_0 + \frac{\ln (1+v)}{v} \right)^{-n}
\end{equation}
where the operator $[v^p]$ represents the coefficient of $v^p$ in a series expansion in $v$.
Comparing \eqref{Def:v},\eqref{Sum:t} and \eqref{Sum:v} leads to a formula for the coefficients $a_n$, which after applying the binomial theorem becomes
$$
a_n=\frac{\omega_0}{n(1+\omega_0)^n}[v^{n-1}] \sum_{k=0}^\infty (-1)^k
\dbinom {n-1+k}{n-1}
\frac{[\ln(1+v)-v]^k}{v^k(1+\omega_0)^k}
$$
or by \eqref{Def:d}
\begin{equation}\label{Coeff:d}
a_n=\frac{\omega_0}{n!} \sum_{k=0}^{n-1} \frac{ (-1)^{n+k-1}\thinspace d(n+k-1,k) }{ (1+\omega_0)^{n+k} }.
\end{equation}
If instead of function \eqref{Def:v} to take
\begin{equation}\label{Def:h}
h=h(t)=W(e^t)-\omega_0-t
\end{equation}
and apply the Lagrange Inversion Theorem to invert a relation
$$
t=\omega_0(e^{-h}-1)-h
$$
coming from \eqref{OmegaRelation}, then we find in a similar way
\begin{equation} \label{Coeff:2assoc}
a_n=\frac{1}{n!} \sum_{k=0}^{n-1}
\StirSubSet{n+k-1}{k}
\frac{ (-1)^{k+1}\omega_0^k } { (1+\omega_0)^{n+k} }.
\end{equation}

Finally, one more representation for the coefficients $a_n$ can be found in the following way. Let us consider a function
\begin{equation} \label{Def:psi}
\psi=\psi(t)=W(e^t)-t
\end{equation}
which is a simplified version of functions \eqref{Def:v} and \eqref{Def:h}: now one does not need to provide the zero function value at $t=0$ and here $\psi(0)=\omega_0$. Then it follows from \eqref{OmegaRelation} that
\begin{equation} \label{eq:tpsi}
t=e^{-\psi}-\psi.
\end{equation}
This equation can also be obtained from the fundamental relation (\ref{eq:basic}) by transformation $u=\psi+\ln t, \sigma=1/t, \tau=\ln t/t$ which follow from \eqref{eq:st}.

Differentiating \eqref{eq:tpsi} in $t$ and excluding the term $e^{-\psi}$ from the result again using \eqref{eq:tpsi} result in an initial value problem for ordinary differential equation
$$
\frac{d\psi}{dt}=-\frac{1}{1+t+\psi}.
$$
Searching a solution in the form of series
\begin{equation} \label{Sum:psi}
\psi(t)=\omega_0+\sum_{n=1}^\infty c_n t^n
\end{equation}
by substituting it into the differential equation and equating coefficients of the same power in $t$ one can finally find

\begin{subequations}\label{2eqs:ac}
\begin{equation} \label{Coeff:c}
c_1=-\frac{1}{1+\omega_0}, \hspace{3mm} c_n=-\frac{1}{n(1+\omega_0)}\left( (n-1)c_{n-1}+\sum_{k=1}^{n-1} k c_k c_{n-k} \right), n=2,3,4,\dots
\end{equation}
At length combining \eqref{Sum:psi},\eqref{Def:psi} and \eqref{Sum:t} gives
\begin{equation} \label{Coeff:ac}
a_1 = 1 + c_1 , \hspace{3mm}  a_n = c_n \mbox{ for } n \geq 2.
\end{equation}
\end{subequations}
In practice computing the expansion coefficients in \eqref{Sum:lnx} using recurrence \eqref{Coeff:c} is faster and takes less digits to obtain a required level of accuracy than using either of \eqref{Coeff:Euler}, \eqref{Coeff:d} or \eqref{Coeff:2assoc} which, however, being different representations of the same expansion coefficients, lead to some combinatorial relations considered in Section \ref{C-R}.

%--------------------------------------------------------------------------------------------------------------------
\subsection{Convergence properties} \label{Wrightfunction}
%--------------------------------------------------------------------------------------------------------------------
The expansion \eqref{eq:wEuler} in fact represents a series of the Wright $\omega$ function \cite{{Corless1997seqseries},{Wright function}} $\omega(z)=W_{\mathcal{K}(z)}(e^z)$, where $\mathcal{K}(z)=\left\lceil (\Im z-\pi)/(2\pi)\right\rceil$ is the unwinding number of $z$. The Wright $\omega$ function was introduced by Corless and Jeffrey \cite{Wright function} and studied in \cite{Wright,Wright function}. When $z \neq \xi \pm i\pi$ for $\xi\leq-1$, $\omega=\omega(z)$  satisfies equation $f(z,\omega)=0$ where $f(z,\omega)=\omega+\ln \omega-z$ (cf. \eqref{OmegaRelation}). Applying the same approach as in Section \ref{RealSigma} to this equation one can obtain the same results as in \cite{{Wright},{Corless1997seqseries},{Wright function}}. Specifically, the nearest to the origin singularities are \cite{Corless1997seqseries}
\begin{equation} \label{eqs:z1z2}
z_1=-1-i\pi \mbox{  and  } z_2=-1+i\pi.
\end{equation}
Note that they are connected with the singularities \eqref{singtau} of function $u=u_\sigma(\tau)$, defined by \eqref{eq:basic} or \eqref{eq:u(sigma,tau)}, through function \eqref{eq:s(sigma,tau)} (cf. Remark \ref{Remark5.2})
$$
z_1=s(\sigma,\tau_\ast^{(1)}) \mbox{  and  } z_2=s(\sigma,\tau_\ast^{(0)}).
$$
Thus the radius of convergence is $\sqrt{1+\pi^2}$ \cite{Corless1997seqseries} and the domain of convergence is defined by
\begin{equation} \label{ineq:sigma}
\left|\sigma\right|>\frac{1}{\sqrt{1+\pi^2}}.
\end{equation}
The estimation of $\omega$ in the vicinity of the singularities \eqref{eqs:z1z2} is \cite{{Wright},{Wright function}}
$$
\omega \sim -1 - \mbox{sgn}(\Im z_k)\sqrt{2z_k}\left(1-\frac{z}{z_k}\right)^\frac{1}{2} \mbox{ as } z\rightarrow z_{k}, \hspace{5mm} (k=1,2)
$$
As in Section \ref{AsymptCoeff}, using the Darboux's theorem one can find the asymptotic expression for the expansion coefficients in \eqref{Sum:lnx}
\begin{equation} \label{eq:asympt}
a_n \sim \sqrt{\frac{2}{\pi}} \frac{(-1)^n\sin \left(\displaystyle\frac{2n-1}{2}\arctan\pi\right)} {n^{\frac{3}{2}}(1+\pi^2)^{\frac{2n-1}{4}} }  \mbox{ as } n\rightarrow\infty .
\end{equation}
Thus, as in case of the series \eqref{eq:improvedStir} for positive $\sigma$ (see \eqref{asympt2}), the expansion coefficients in the series \eqref{eq:wEuler} disclose decaying oscillations in their behavior for large $n$.

In real case inequality \eqref{ineq:sigma} read as $\exp(-\sqrt{1+\pi^2})<x<\exp(\sqrt{1+\pi^2})$. Thus from the point of view of the domain of convergence the series \eqref{eq:wEuler} takes an intermediate place between the expansion of $W(x)$ at the origin \cite{Corless1996Lam} $W(x)=\sum_{n=1}^\infty (-n)^{n-1}x^n/n!$, which is valid for $-e^{-1}<x<e^{-1}$, and the series \eqref{eq:improvedStir} which is valid for $x>x_1=1.004458...$ (see Corollary \ref{cor:5.1}). These three expansions put together cover the entire region of definition of $W(x)$.

%-------------------------------------------------------------------------------------------------------------------------
\section{Combinatorial consequences} \label{C-R}
%-------------------------------------------------------------------------------------------------------------------------
In this section we collect some combinatorial consequences resulting from the above obtained expressions for the expansion coefficients.

Equating the right-hand sides of \eqref{dzetaStir2} and \eqref{dzetaEuler2} one can find

\begin{subequations} \label{CRident}
\begin{equation} \label{Euler2assoc:lambda}
\sum_{k=0}^{n}
\EulerianTwo{n}{k}(1+\lambda)^{n-k-1}\lambda^k
=
\sum_{k=1}^n
\StirSubSet{n+k}{k} \lambda^{k-1},
\end{equation}
where summation in the right-hand side starts with one as $\genfrac{\{}{\}}{0pt}{}{n}{0}_{\negthickspace\scriptscriptstyle\geq2}=0$ \cite{Graham et. al}.
Setting $\mu=\lambda/(1-\lambda)$ in \eqref{Euler2assoc:lambda} we also find
\begin{equation} \label{Euler2assoc:mu}
\sum_{k=0}^{n}
\EulerianTwo{n}{k} \mu^k
=
\sum_{k=1}^n
\StirSubSet{n+k}{k} \mu^{k-1} (1-\mu)^{m-k}.
\end{equation}
\end{subequations}

The identities \eqref{CRident} were obtained by L.M. Smiley in a different way in \cite{Smiley 1}, where notation
$\left\{\left\{\right\}\right\}$ was used instead of $\StirSubSet{}{}$, and referred to as the Carlitz-Riordan identities \cite{Smiley 2}.
Applying the binomial theorem to \eqref{CRident} leads to a pair of identities expressing the 2-associated Stirling numbers of the second kind through
the second-order Eulerian numbers and inversely \cite{Smiley 1}
$$
\StirSubSet{n+q}{q} = \sum_{p=0}^n \dbinom{n-p-1}{q-p-1}\EulerianTwo{n}{p}
$$
\vspace{2mm}
$$
\EulerianTwo{n}{q} = \sum_{p=0}^n (-1)^{q-p}\dbinom{n-p-1}{q-p}\StirSubSet{n+p+1}{p+1}
$$

Some estimates can also be obtained by comparing the found asymptotic formulas \eqref{asympt2} and \eqref{asympt3} with the explicit expressions for the expansion coefficients in \eqref{eq:improvedStir}. For example, taking estimate \eqref{asympt3} and the expansion coefficients in \eqref{eq:improvedStir} at $\sigma=-2$ we obtain
\begin{equation} \label{asympt:2assoc}
\sum_{p=1}^{m-1} \StirSubSet{p+m-1}{p} \sim
\frac{(m-1)!}{2\sqrt{\pi m} \left(2\ln2-1 \right)^{m-\frac{1}{2}}}  \hspace{5mm} \mbox{ as } \hspace{2mm} m\rightarrow\infty
\end{equation}
where the term with $p=0$ is skipped (cf. \eqref{Euler2assoc:lambda}. This result is consistent with the formula given in \cite[Ex.10(7), p.\,224]{Comtet}.

Another consequence is obtained by taking the expansion coefficients in \eqref{eq:improvedStir} at $\sigma=0$ together with \eqref{asympt4}
\begin{equation} \label{sum:2assoc}
\sum_{p=0}^{m-1}(-1)^{p+m-1}\StirSubSet{p+m-1}{p} = (m-1)!
\end{equation}

Further, comparing \eqref{Coeff:Euler}, \eqref{Coeff:d} and \eqref{Coeff:2assoc} between one another we come to the following three identities
\begin{subequations} \label{Euler-d-2assoc}
\begin{equation}
\displaystyle\frac{1}{(1+\omega_0)^{n-1} }
\sum_{k=0}^{n-1}\EulerianTwo{n-1}{k}
(-1)^k \omega_0^k = \sum_{k=0}^{n-1} \frac{ (-1)^{n+k-1}\thinspace d(n+k-1,k) }{ (1+\omega_0)^k }
\end{equation}
\par\medskip
\begin{equation}
\displaystyle \frac{1}{ (1+\omega_0)^{n-1} }
\sum_{k=0}^{n-1} \EulerianTwo{n-1}{k}(-1)^{k+1} \omega_0^{k+1}
=
\sum_{k=0}^{n-1}
\StirSubSet{n+k-1}{k}\frac{ (-1)^k \omega_0^k } { (1+\omega_0)^k }
\end{equation}
\par\medskip
\begin{equation}
\omega_0\sum_{k=0}^{n-1} \frac{ (-1)^{n+k}\thinspace d(n+k-1,k) }{ (1+\omega_0)^k }
=
\sum_{k=0}^{n-1}
\StirSubSet{n+k-1}{k}\frac{ (-1)^k \omega_0^{k} } {(1+\omega_0)^k}
\end{equation}
\end{subequations}
\\
Finally, combining either of \eqref{Coeff:Euler}, \eqref{Coeff:d} or \eqref{Coeff:2assoc} with \eqref{eq:asympt} gives an asymptotic expression for the sum involved there.

Thus studying expansion series for $W$ function we, on the way,  derived the Carlitz-Riordan identities \eqref{CRident} as well as found a formula for an alternating sum of 2-associated Stirling numbers of the second kind \eqref{sum:2assoc} and confirmed the asymptotic formula \eqref{asympt:2assoc} for summation of the same numbers without the alternating factor. We also found formulas \eqref{Euler-d-2assoc} where the Omega constant $\omega_0$ plays a role of a `magic' number which connects sums involving the second-order Eulerian numbers, the associated Stirling numbers of the first kind and the 2-associated Stirling numbers of the second kind.

%-------------------------------------------------------------------------------------------------------------------------
\section{Concluding remarks}
%-------------------------------------------------------------------------------------------------------------------------
We ascertained the domain of convergence of the series \eqref{eq:Comtet} and \eqref{eq:improvedStir} in real and complex cases and found that the series \eqref{eq:improvedStir} has a much wider domain of convergence than that of the series \eqref{eq:Comtet} in both cases and provided an analysis of this fact in real case. We found asymptotic expressions for the expansion coefficients and obtained a representation of the series \eqref{eq:improvedStir} in terms of the second-order Eulerian numbers.

We considered an invariant transformation defined by the parameter $p$ and applied it to the above series. We studied an effect of parameter $p$ on the convergence properties of the transformed series theoretically and numerically and found that an increase of $p$ results in an extension of the domain of convergence of the series. Thus the series obtained under the transformation with positive values of $p$ have a wider domain of convergence than the original series does. However, at the same time a rate of convergence can be found to be reduced when the parameter $p$
increases. Therefore in such a case within the common domain of convergence of the series with different positive values of $p$ the series with the minimum value of $p$ would be the most effective.

We also considered the well-known expansion of $W(x)$ in powers of $\ln x$  and gave an asymptotic estimate for the expansion coefficients. We found three more forms for a representation of the expansion coefficients of the series in terms of the associated Stirling numbers of the first kind (\ref{Coeff:d}), the 2-associated Stirling subset numbers (\ref{Coeff:2assoc}) and iterative formulas \eqref{2eqs:ac}. Finally we presented some combinatorial consequences, including the Carlitz-Riordan identities, which result from the found different forms of the expansion coefficients of the above series.

%%%%%%%%%%%%%%%%%%%%%%%%%%%%%%%%%%%%%%%%%%%%%%%%%%%%%%%%%%%%%%%%%%%%%%%%%%%%%%%%%%%%%%%%%%%%%%%%%%%%%%%%%%%%%%%%%%%%%%%%%%

%%%%%%%%%%%%%%%%%%%%%%%%%%%%%%%%%%%%%%%%%%%%%%%%%%%%%%%%%%%%%%%%%%%%%%%%%%%%%%%%%%%%%%%%%%%%%%%%%%%%%%%%%%%%%%%%%%%%%
\end{document}